\documentclass[11pt,reqno]{amsart}

\usepackage[utf8]{inputenc}
\usepackage[T1]{fontenc}
\usepackage{lmodern}
\usepackage{amsmath,amssymb,amsthm,mathrsfs}
\usepackage{mathtools}
\usepackage{geometry}
\usepackage{enumitem}
\usepackage[colorlinks=true,linkcolor=blue,citecolor=red,urlcolor=blue]{hyperref}
\usepackage{microtype}
\usepackage{comment}
\hypersetup{
  pdftitle={Rhaly operators between Dirichlet-type spaces: a complete classification},
  pdfauthor={Yecheng Shi}
}

\theoremstyle{plain}
\newtheorem{theorem}{Theorem}[section]
\newtheorem{lemma}[theorem]{Lemma}
\newtheorem{proposition}[theorem]{Proposition}
\newtheorem{corollary}[theorem]{Corollary}

\theoremstyle{definition}
\newtheorem{definition}[theorem]{Definition}

\theoremstyle{remark}
\newtheorem{remark}[theorem]{Remark}

\newcommand{\D}{\mathbb D}
\newcommand{\T}{\mathbb T}
\newcommand{\N}{\mathbb N}
\newcommand{\Hol}{\operatorname{Hol}}

\title[Rhaly operators between Dirichlet-type spaces]
{Rhaly operators between Dirichlet-type spaces:
a complete classification}
\author{Yecheng Shi}
\address{School of Mathematics and Statistics, Lingnan Normal University,
Zhanjiang 524048, Guangdong, China}
\email{09ycshi@sina.cn}

\begin{document}
\begin{abstract}
Let \(\eta=(\eta_n)_{n\ge0}\) be a complex sequence such that
\(F_\eta(z)=\sum_{n\ge0}\eta_nz^n\in\Hol(\D)\), and let
\(\mathcal R_{(\eta)}\) be the associated Rhaly operator.
We give a complete characterization of boundedness and compactness of
\(\mathcal R_{(\eta)}:\mathcal D^p_\alpha\to\mathcal D^q_\beta\) for
\(1<p,q<\infty\) and \(\alpha,\beta>-1\).
The characterization is formulated in terms of the \(H^q\)-norms of
the dyadic blocks of \(F_\eta\). For \(-1<\alpha<p-2\), boundedness
and compactness coincide and reduce to
\(F_\eta\in\mathcal D^q_\beta\). For \(\alpha=p-2\) and for
\(\alpha>p-2\), boundedness is characterized by membership in analytic
truncated Besov spaces and analytic Besov spaces, respectively;
compactness has the same characterization for \(q<p\) and is
characterized by the corresponding little spaces for \(p\le q\).
We also obtain operator norm and essential norm estimates.
As applications, we obtain boundedness and compactness
characterizations for Rhaly operators between weighted Bergman spaces
and for Ces\`aro-type operators induced by positive measures. In particular, we characterize \(C_\mu:B^p\to B^p\), answering a
question left open by Sun, Ye and Zhou and later noted by Tang.
For each \(p>2\), we also
construct a symbol \(F_\eta\in H^\infty\cap\lambda^p_{1/p}\) for which
\(\mathcal R_{(\eta)}\) is not bounded on \(H^p\).
\end{abstract}
\maketitle

\section{Introduction}
\label{sec:introduction}

Let \(\eta=(\eta_n)_{n\ge0}\) be a complex sequence such that
\(F_\eta(z)=\sum_{n\ge0}\eta_nz^n\in\Hol(\D)\). The associated Rhaly
operator is initially defined on analytic polynomials
\(f(z)=\sum_{k\ge0}a_kz^k\) by
\[
        \mathcal R_{(\eta)}f(z)
        =
        \sum_{n=0}^{\infty}
        \eta_n\left(\sum_{k=0}^{n}a_k\right)z^n.
\]
When \(\eta_n=(n+1)^{-1}\), this is the classical Ces\`aro operator.
If \(\mu\) is a finite complex Borel measure on \([0,1)\), then
\(C_\mu=\mathcal R_{(\mu_n)}\), where
\(\mu_n=\int_{[0,1)}t^n\,d\mu(t)\). On Taylor coefficients, \(\mathcal R_{(\eta)}\) is represented by a
Rhaly matrix, also known as a terraced matrix. The terminology goes
back to Leibowitz \cite{Leibowitz1987}, while Rhaly studied these
matrices under the name terraced matrices \cite{Rhaly1989}.
More recent results concern boundedness and compactness on
\(\ell^p\) spaces \cite{GalanopoulosGirelaPrajitura2024}, Schatten
class membership \cite{BellavitaDellepianeStylogiannis2025}, and
spectral properties and invariant subspaces
\cite{GallardoGutierrezPartington2025}.

For \(1<p<\infty\) and \(\alpha>-1\), let \(A^p_\alpha\) be the
weighted Bergman space of all \(f\in\Hol(\D)\) such that
\(\|f\|_{A^p_\alpha}^p
=(\alpha+1)\int_{\D}|f(z)|^p(1-|z|^2)^\alpha\,dA(z)<\infty\),
where \(dA\) denotes normalized area measure on \(\D\).
The Dirichlet-type space \(\mathcal D^p_\alpha\) consists of the
functions \(f\in\Hol(\D)\) such that \(f'\in A^p_\alpha\), with
\[
        \|f\|_{\mathcal D^p_\alpha}^p
        =
        |f(0)|^p+\|f'\|_{A^p_\alpha}^p.
\]
We write \(B^p=\mathcal D^p_{p-2}\) for the classical analytic Besov
space. We also use the analytic Besov spaces \(B^s_{q,u}\), their polynomial
closures \(b^s_{q,\infty}\), the analytic truncated Besov spaces
\(T_u^bB^s_{q,q}\), the little analytic truncated Besov spaces
\(t_u^bB^s_{q,q}\), and the mean-Lipschitz spaces
\(\Lambda_\gamma^t\) and \(\lambda_\gamma^t\).
The spaces \(B^s_{q,u}\), \(b^s_{q,\infty}\),
\(\Lambda_\gamma^t\), and \(\lambda_\gamma^t\) are recalled in
Subsection~\ref{subsec:analytic-besov}, while
\(T_u^bB^s_{q,q}\) and \(t_u^bB^s_{q,q}\) are defined in
Subsection~\ref{subsec:analytic-truncated-besov}.

The classical Ces\`aro operator on Hardy spaces was studied by
Siskakis \cite{Siskakis1987,Siskakis1990} and Miao
\cite{Miao1992}, and on Bergman spaces by Siskakis
\cite{Siskakis1996}. Stempak \cite{Stempak1994} introduced a family
of Ces\`aro averaging operators and studied their action on Hardy
spaces. Andersen subsequently studied these operators on Hardy spaces
\cite{Andersen1996} and Dirichlet spaces \cite{Andersen2004}.
Xiao \cite{Xiao1997} studied Ces\`aro-type operators on Hardy spaces,
\(\mathrm{BMOA}\), and Bloch spaces. For finite positive Borel measures \(\mu\) on \([0,1)\),
Galanopoulos, Girela and Merch\'an
\cite{GalanopoulosGirelaMerchan2022} studied \(C_\mu\) on Hardy and
weighted Bergman spaces, \(\mathrm{BMOA}\), and the Bloch space.
Jin and Tang \cite{JinTang2022} characterized boundedness and
compactness of
\(\mathcal C_\mu:\mathcal D^2_\alpha\to\mathcal D^2_\beta\) for
\(0<\alpha,\beta<2\). Galanopoulos, Girela, Mas and Merch\'an
\cite{GalanopoulosGirelaMasMerchan2023} studied the operators
\(\mathcal H_\mu\) and \(C_\mu\) on the classical Dirichlet space and,
more generally, on analytic Besov spaces. Sun, Ye and Zhou \cite{SunYeZhou2025} characterized the boundedness
and compactness of $C_\mu:B^p\to X$ for Banach spaces $X$ satisfying
$\Lambda^s_{1/s}\subset X\subset\mathcal B$, $p,s>1$, and left open
the characterization of $C_\mu:B^p\to B^p$.
Tang \cite{Tang2025Besov} subsequently noted the corresponding open
problem for the generalized operators
$\mathcal C_{\mu,\alpha}:B^p\to B^p$; the case $\alpha=1$ is $C_\mu$. Galanopoulos, Siskakis and Zhao
\cite{GalanopoulosSiskakisZhao2025} characterized boundedness of the
weighted Ces\`aro-type operators \(\mathcal C_{\mu,\beta}\) from
\(A_\alpha^p\) to \(A_\alpha^q\) for \(1\le p\le q<\infty\).
Blasco and Mas \cite{BlascoMas2026} studied boundedness of
Ces\`aro-type operators between mixed norm spaces and obtained, in
particular, characterizations between weighted Bergman spaces with
different weights. Xie, Liu and Lin
\cite{XieLiuLin2026} studied boundedness and compactness of
\(C_\mu:\mathcal D^p_\alpha\to\mathcal D^q_\beta\); the corrected
statement for \(p>\alpha+2\) is given in
\cite[Theorem~1.2]{LinLiuTangXie2026}. Guo and Tang \cite{GuoTang2026} characterized the boundedness and
compactness of $C_\mu:F(p,p-1,1)\to\mathcal D^p_{p-1}$ for $1<p<\infty$,
and obtained corresponding results from $F(p,q,s)$ into Bloch-type
spaces in the range $s>0$ and $q+2>p$.

Complex measures and arbitrary complex sequences were considered more
recently. Galanopoulos, Girela and Merch\'an
\cite{GalanopoulosGirelaMerchan2023} studied Ces\`aro-type operators
associated with complex Borel measures on \(H^2\) and weighted
Bergman spaces. Blasco studied Ces\`aro-type operators associated with
complex Borel measures on Hardy spaces
\cite{Blasco2024Hardy} and weighted Dirichlet spaces
\cite{Blasco2024WeightedDirichlet}. Bao, Guo, Sun and Wang
\cite{BaoGuoSunWang2024} characterized boundedness and compactness of
\(\mathcal R_{(\eta)}\) on the classical Dirichlet space for
arbitrary complex sequences \(\eta\). Galanopoulos and Girela
\cite{GalanopoulosGirela2025HilbertCesaro} studied Ces\`aro-type
operators induced by complex sequences on
\(\mathcal D_\alpha^2\), \(0\le\alpha\le1\). Lin and Xie
\cite{LinXie2025} studied Ces\`aro-type operators on derivative-type
Hilbert spaces and obtained, in particular, complete characterizations
of boundedness and compactness between different weighted Bergman
spaces. Blasco, Galanopoulos and Girela
\cite{BlascoGalanopoulosGirela2026} characterized
\(\mathcal R_{(\eta)}:\mathcal D_\alpha^2\to\mathcal D_\beta^2\)
for arbitrary complex sequences \(\eta\) and
\(\alpha,\beta\in\mathbb R\). Galanopoulos and Girela
\cite{GalanopoulosGirela2026} obtained boundedness and compactness
conditions involving \(\Lambda^p_{1/p}\) and \(\lambda^p_{1/p}\) on
Hardy, Bergman and Dirichlet spaces, with complete characterizations
on \(H^2\), on \(A^2_\alpha\) for \(\alpha>-1\), and on
\(\mathcal D^2_\alpha\) for \(\alpha>0\). Bao, Tian and Wulan
\cite{BaoTianWulan2026} characterized boundedness and compactness of
Ces\`aro-type operators induced by complex sequences from \(H^p\),
\(0<p<\infty\), and from Banach spaces between
\(\Lambda^2_{1/2}\) and the Bloch space into
\(\Lambda^2_{1/2}\). Bellavita, Dellepiane and Stylogiannis
\cite{BellavitaDellepianeStylogiannis2026} studied boundedness,
compactness and Schatten class membership of Rhaly operators on
weighted Hardy spaces \(H^2(\omega)\). Galanopoulos and Girela
\cite{GalanopoulosGirela2026b} later characterized boundedness and
compactness of
\(\mathcal R_{(\eta)}:\mathcal D_a^p\to\mathcal D_b^p\) in the range
\(p-2<a<\min\{p-1+b,2p-2\}\), \(a-b+1>0\), and of
\(\mathcal R_{(\eta)}:A^p\to A^q\) when \(1<p<q<\infty\) and
\(2/p-1/q<1\).

In this paper, we characterize boundedness and compactness of
\(\mathcal R_{(\eta)}:\mathcal D^p_\alpha\to\mathcal D^q_\beta\)
for arbitrary complex symbols, \(1<p,q<\infty\), and
\(\alpha,\beta>-1\). The characterization is formulated in terms of
the \(H^q\)-norms of the dyadic blocks of \(F_\eta\). For
\(-1<\alpha<p-2\), boundedness and compactness coincide and are
equivalent to \(F_\eta\in\mathcal D^q_\beta\). When
\(\alpha=p-2\), boundedness and compactness coincide for \(q<p\) and
are characterized by membership in an analytic truncated Besov space;
for \(p\le q\), boundedness is characterized by an analytic truncated
Besov space, while compactness is characterized by a little analytic
truncated Besov space. For \(\alpha>p-2\), boundedness and compactness
coincide for \(q<p\) and are characterized by membership in an
analytic Besov space; for \(p\le q\), boundedness is characterized by
an analytic Besov space, while compactness is characterized by a
little Besov space.

Put \(I_0=\{0,1\}\) and
\(I_j=\{2^j,\ldots,2^{j+1}-1\}\) for \(j\ge1\), and set
\(\Delta_jf(z)=\sum_{n\in I_j}\widehat f(n)z^n\).
For \(1<t<\infty\), write \(t'=t/(t-1)\). We denote by
\(\mathcal B(X,Y)\) and \(\mathcal K(X,Y)\) the spaces of bounded and
compact operators from \(X\) to \(Y\), respectively, and by
\(\|T\|_e\) the essential norm of \(T\). For nonnegative quantities
\(A\) and \(B\), we write \(A\lesssim B\) if \(A\le CB\), where
\(C\) depends only on the fixed parameters, and \(A\asymp B\) if
\(A\lesssim B\) and \(B\lesssim A\).

In Theorems~\ref{thm:intro-below-critical}--\ref{thm:intro-above-critical},
we assume that \(\eta=(\eta_n)_{n\ge0}\) is a complex sequence such that
\(F_\eta\in\Hol(\D)\). For \(-1<\alpha<p-2\), we have the following
characterization.

\begin{theorem}
\label{thm:intro-below-critical}
Let \(1<p,q<\infty\), \(-1<\alpha<p-2\), and \(\beta>-1\). Then
\[
        \mathcal R_{(\eta)}
        \in\mathcal B(\mathcal D^p_\alpha,\mathcal D^q_\beta)
        \Longleftrightarrow
        \mathcal R_{(\eta)}
        \in\mathcal K(\mathcal D^p_\alpha,\mathcal D^q_\beta)
        \Longleftrightarrow
        F_\eta\in\mathcal D^q_\beta
        \Longleftrightarrow
        F_\eta\in B^{1-(\beta+1)/q}_{q,q}.
\]
Moreover,
\[
        \|\mathcal R_{(\eta)}\|_{\mathcal D^p_\alpha\to\mathcal D^q_\beta}
        \asymp_{p,q,\alpha,\beta}
        \|F_\eta\|_{\mathcal D^q_\beta}
        \asymp_{p,q,\alpha,\beta}
        \|F_\eta\|_{B^{1-(\beta+1)/q}_{q,q}}.
\]
\end{theorem}

For \(1<p,q<\infty\) and \(\beta>-1\), put
\(W_J(\eta)=\sum_{j\ge J}2^{j(q-\beta-1)}
\|\Delta_jF_\eta\|_{H^q}^q\), \(J\ge0\).
If \(\alpha=p-2\) and \(q<p\), put \(1/r=1/q-1/p\) and set
\[
        G_{p,q,\beta}(\eta)
        =
        \left(
        \sum_{J\ge0}
        (J+1)^{r/q'}W_J(\eta)^{r/q}
        \right)^{q/r}.
\]
For \(\alpha=p-2\), we have the following characterization.

\begin{theorem}
\label{thm:intro-critical}
Let \(1<p,q<\infty\), \(\beta>-1\), and put
\(
s=1-\frac{\beta+1}{q}.
\)
\begin{enumerate}[label=\textup{(\roman*)},leftmargin=2.2em]

\item If \(q<p\) and \(1/r=1/q-1/p\), then
\[
\begin{aligned}
\mathcal R_{(\eta)}
\in\mathcal B(\mathcal D^p_{p-2},\mathcal D^q_\beta)
&\Longleftrightarrow
\mathcal R_{(\eta)}
\in\mathcal K(\mathcal D^p_{p-2},\mathcal D^q_\beta) \\
&\Longleftrightarrow
G_{p,q,\beta}(\eta)<\infty
\Longleftrightarrow
F_\eta\in T_r^{1/p'}B^s_{q,q}.
\end{aligned}
\]
Moreover,
\[
\|\mathcal R_{(\eta)}\|_{\mathcal D^p_{p-2}\to\mathcal D^q_\beta}
\asymp_{p,q,\beta}
G_{p,q,\beta}(\eta)^{1/q}
\asymp_{p,q,\beta}
\|F_\eta\|_{T_r^{1/p'}B^s_{q,q}}.
\]

\item If \(p\le q\), then
\[
\mathcal R_{(\eta)}
\in\mathcal B(\mathcal D^p_{p-2},\mathcal D^q_\beta)
\Longleftrightarrow
\sup_{J\ge0}(J+1)^{q/p'}W_J(\eta)<\infty
\Longleftrightarrow
F_\eta\in T_\infty^{1/p'}B^s_{q,q},
\]
and
\[
\|\mathcal R_{(\eta)}\|_{\mathcal D^p_{p-2}\to\mathcal D^q_\beta}
\asymp_{p,q,\beta}
\sup_{J\ge0}(J+1)^{1/p'}W_J(\eta)^{1/q}
\asymp_{p,q,\beta}
\|F_\eta\|_{T_\infty^{1/p'}B^s_{q,q}}.
\]
Moreover,
\[
\mathcal R_{(\eta)}
\in\mathcal K(\mathcal D^p_{p-2},\mathcal D^q_\beta)
\Longleftrightarrow
\lim_{J\to\infty}(J+1)^{q/p'}W_J(\eta)=0
\Longleftrightarrow
F_\eta\in t_\infty^{1/p'}B^s_{q,q}.
\]
If \(\mathcal R_{(\eta)}\) is bounded, then
\[
\begin{aligned}
\|\mathcal R_{(\eta)}\|_e
&\asymp_{p,q,\beta}
\operatorname{dist}_{T_\infty^{1/p'}B^s_{q,q}}
\bigl(F_\eta,t_\infty^{1/p'}B^s_{q,q}\bigr) \\
&\asymp_{p,q,\beta}
\limsup_{k\to\infty}
2^{k/p'}
\left(
\sum_{j=2^k-1}^{2^{k+1}-2}
2^{jsq}\|\Delta_jF_\eta\|_{H^q}^q
\right)^{1/q} \\
&\asymp_{p,q,\beta}
\limsup_{J\to\infty}
(J+1)^{1/p'}W_J(\eta)^{1/q}.
\end{aligned}
\]
\end{enumerate}
\end{theorem}

For \(\alpha>p-2\), we have the following characterization.

\begin{theorem}
\label{thm:intro-above-critical}
Let \(1<p,q<\infty\), \(\alpha>p-2\), \(\beta>-1\), and put
\(
s=\frac{\alpha+2}{p}-\frac{\beta+1}{q}.
\)
\begin{enumerate}[label=\textup{(\roman*)},leftmargin=2.2em]

\item If \(q<p\) and \(1/r=1/q-1/p\), then
\[
\begin{aligned}
\mathcal R_{(\eta)}
\in\mathcal B(\mathcal D^p_\alpha,\mathcal D^q_\beta)
&\Longleftrightarrow
\mathcal R_{(\eta)}
\in\mathcal K(\mathcal D^p_\alpha,\mathcal D^q_\beta) \\
&\Longleftrightarrow
\bigl(2^{js}\|\Delta_jF_\eta\|_{H^q}\bigr)_{j\ge0}\in\ell^r \\
&\Longleftrightarrow
F_\eta\in B^s_{q,r}.
\end{aligned}
\]
Moreover,
\[
\|\mathcal R_{(\eta)}\|_{\mathcal D^p_\alpha\to\mathcal D^q_\beta}
\asymp_{p,q,\alpha,\beta}
\left\|
\bigl(2^{js}\|\Delta_jF_\eta\|_{H^q}\bigr)_{j\ge0}
\right\|_{\ell^r}
\asymp_{p,q,\alpha,\beta}
\|F_\eta\|_{B^s_{q,r}}.
\]

\item If \(p\le q\), then
\[
\mathcal R_{(\eta)}
\in\mathcal B(\mathcal D^p_\alpha,\mathcal D^q_\beta)
\Longleftrightarrow
\sup_{j\ge0}2^{js}\|\Delta_jF_\eta\|_{H^q}<\infty
\Longleftrightarrow
F_\eta\in B^s_{q,\infty},
\]
and
\[
\|\mathcal R_{(\eta)}\|_{\mathcal D^p_\alpha\to\mathcal D^q_\beta}
\asymp_{p,q,\alpha,\beta}
\sup_{j\ge0}2^{js}\|\Delta_jF_\eta\|_{H^q}
\asymp_{p,q,\alpha,\beta}
\|F_\eta\|_{B^s_{q,\infty}}.
\]
Moreover,
\[
\mathcal R_{(\eta)}
\in\mathcal K(\mathcal D^p_\alpha,\mathcal D^q_\beta)
\Longleftrightarrow
\lim_{j\to\infty}
2^{js}\|\Delta_jF_\eta\|_{H^q}=0
\Longleftrightarrow
F_\eta\in b^s_{q,\infty}.
\]
If \(\mathcal R_{(\eta)}\) is bounded, then
\[
\|\mathcal R_{(\eta)}\|_e
\asymp_{p,q,\alpha,\beta}
\operatorname{dist}_{B^s_{q,\infty}}
\bigl(F_\eta,b^s_{q,\infty}\bigr) 
\asymp_{p,q,\alpha,\beta}
\limsup_{j\to\infty}
2^{js}\|\Delta_jF_\eta\|_{H^q}.
\]
\end{enumerate}
\end{theorem}

Taken together, Theorems~\ref{thm:intro-below-critical}--\ref{thm:intro-above-critical}
characterize boundedness and compactness of
\(\mathcal R_{(\eta)}:\mathcal D^p_\alpha\to\mathcal D^q_\beta\)
for \(1<p,q<\infty\) and \(\alpha,\beta>-1\), and provide
corresponding operator norm estimates. Whenever \(q<p\),
boundedness implies compactness. For \(p\le q\), boundedness
and compactness coincide when \(-1<\alpha<p-2\), but may differ
when \(\alpha\ge p-2\). For \(p\le q\), compactness is
characterized by \(F_\eta\in t_\infty^{1/p'}B^s_{q,q}\) when
\(\alpha=p-2\), and by \(F_\eta\in b^s_{q,\infty}\) when
\(\alpha>p-2\), with \(s\) as defined in the respective theorems.
For \(p\le q\) and \(\alpha\ge p-2\), the essential norm of a
bounded operator is comparable to the distance of \(F_\eta\)
to the little space in the corresponding compactness criterion.

Taking \(q=p\) and \(\beta=p-2\) in
Theorem~\ref{thm:intro-critical}\textup{(ii)} gives boundedness
and compactness criteria, together with an essential norm estimate,
for \(\mathcal R_{(\eta)}\) on
\(B^p=\mathcal D^p_{p-2}\). For \(p=2\), the boundedness and
compactness criteria recover the known tail conditions on the
classical Dirichlet space
\cite{BaoGuoSunWang2024,BlascoGalanopoulosGirela2026}.
More generally, for \(p=q=2\), the three theorems recover the
\(\alpha,\beta>-1\) part of the characterization of Blasco,
Galanopoulos and Girela
\cite{BlascoGalanopoulosGirela2026}.

Theorem~\ref{thm:intro-above-critical} also covers the parameter
ranges studied in \cite[Theorems~3--4]{GalanopoulosGirela2026b}.
When \(q=p\), it recovers their characterization of
\(\mathcal R_{(\eta)}:\mathcal D_a^p\to\mathcal D_b^p\) for
\(p-2<a<\min\{p-1+b,2p-2\}\) and \(a-b+1>0\), with
\(s=(a-b+1)/p\). For unweighted Bergman spaces,
\(A^u\simeq\mathcal D_u^u\) when \(1<u<\infty\).
Taking \(\alpha=p\) and \(\beta=q\) gives
\(s=2/p-1/q\) for \(1<p<q<\infty\) and \(2/p-1/q<1\).
This agrees with the exponent used in the proof of
\cite[Theorem~3]{GalanopoulosGirela2026b}.
More generally,
Theorems~\ref{thm:intro-below-critical}--\ref{thm:intro-above-critical}
apply to \(1<p,q<\infty\) and \(\alpha,\beta>-1\), including
\(p\ne q\) for Dirichlet-type spaces.

Corollary~\ref{cor:same-space-classification} specializes the
results to operators on \(\mathcal D^p_\alpha\), while
Corollary~\ref{cor:mixed-Bergman} treats pairs of weighted
Bergman spaces and gives essential norm estimates when
\(p\le q\). For a finite positive Borel measure \(\mu\) on
\([0,1)\), Corollary~\ref{cor:cesaro-positive-moment}
translates the dyadic criteria into conditions on the moment
sequence \((\mu_n)\). In particular,
Corollary~\ref{cor:cesaro-positive-moment}\textup{(ii)(b)}
gives boundedness, compactness, and essential norm criteria for
\(C_\mu:B^p\to B^p\). This answers the question posed by
Sun, Ye and Zhou \cite{SunYeZhou2025}; Tang
\cite{Tang2025Besov} later noted the corresponding problem
for generalized Ces\`aro operators. Finally,
Theorem~\ref{thm:endpoint-failure} shows that, for \(p>2\),
even \(F_\eta\in H^\infty\cap\lambda^p_{1/p}\) does not imply
boundedness of \(\mathcal R_{(\eta)}\) on \(H^p\).

For arbitrary complex \(\eta\), the positivity and monotonicity
of positive-measure moments cannot be used. Our proofs rely on
dyadic decompositions and coefficient multiplier estimates.
When \(-1<\alpha<p-2\), the Taylor coefficient partial sums
of each \(f\in\mathcal D^p_\alpha\) define a bounded coefficient
multiplier on \(A^q_\beta\). This gives the criterion
\(F_\eta\in\mathcal D^q_\beta\).

At \(\alpha=p-2\), each block of \(\mathcal R_{(\eta)}f\)
splits into a local triangular term and a term involving
\(\Delta_jF_\eta\) and preceding blocks of \(f\). A bilinear
estimate controls the local term, while weighted discrete Hardy
inequalities give the required bounds. For necessity when
\(q<p\), we combine normalized block-average tests with
coefficient multipliers of uniformly bounded variation.
When \(p\le q\), we use normalized sums of consecutive dyadic
blocks. Corollary~\ref{cor:critical-truncated-besov-norms}
identifies the resulting tail criteria with analytic truncated
Besov conditions.

When \(\alpha>p-2\), sufficiency follows from a discrete
convolution estimate and the characterization of diagonal
multipliers from \(\ell^p\) to \(\ell^q\). For necessity, we
test on polynomials supported on alternating dyadic blocks
when \(q<p\), and on a single block when \(p\le q\). The
resulting boundedness criteria are \(F_\eta\in B^s_{q,r}\)
when \(q<p\), where \(1/r=1/q-1/p\), and
\(F_\eta\in B^s_{q,\infty}\) when \(p\le q\). For
\(\alpha>p-2\) and \(p\le q\), compactness is characterized by
\(F_\eta\in b^s_{q,\infty}\). For \(p\le q\) and \(\alpha\ge p-2\), finite-rank
approximations and weakly null test sequences yield the
essential norm estimates.

Section~\ref{sec:preliminaries} collects the dyadic tools and the
analytic Besov and truncated Besov spaces used in the proofs.
Sections~\ref{sec:below-critical}--\ref{sec:above-critical} treat
\(-1<\alpha<p-2\), \(\alpha=p-2\), and \(\alpha>p-2\),
respectively. Section~\ref{sec:endpoint-regime} gives the weighted
Bergman corollary and the Hardy space counterexample.
Section~\ref{sec:cesaro-type} treats operators induced by positive
measures; the required coefficient estimates are proved in
Appendix~\ref{app:monotone-reductions}.

\section{Preliminaries}
\label{sec:preliminaries}

\subsection{Dyadic decompositions, multipliers, and norm equivalences}

\begin{lemma}
\label{lem:bv-multiplier}
Let \(1<p<\infty\), and suppose that
\[
A_\lambda
:=
\sup_{n\ge0}|\lambda_n|
+
\sum_{n=0}^{\infty}|\lambda_{n+1}-\lambda_n|
<\infty.
\]
Then the coefficient multiplier
\(T_\lambda f=\sum_{n\ge0}\lambda_n\widehat f(n)z^n\)
is bounded on \(H^p\), and
\(\|T_\lambda f\|_{H^p}\le C_pA_\lambda\|f\|_{H^p}\).
\end{lemma}

\begin{proof}
This is the standard bounded variation multiplier theorem; see
\cite[Theorem~12.7.2]{JevticVukoticArsenovic2016}. The stated estimate
also follows from Abel summation and the uniform boundedness of the
Taylor partial sum operators on \(H^p\).
\end{proof}

As an immediate consequence, we have the following interval projection estimate.

\begin{lemma}
\label{lem:Hp-interval-projections}
Let \(1<p<\infty\), let \(J\subset\N_0\) be a finite integer interval,
and set
\(\Pi_Jh=\sum_{n\in J}\widehat h(n)z^n\).
Then
\(\|\Pi_Jh\|_{H^p}\le C_p\|h\|_{H^p}\),
where \(C_p\) is independent of \(J\).
\end{lemma}

\begin{proof}
Apply Lemma~\ref{lem:bv-multiplier} to
\(\lambda_n=\mathbf1_J(n)\), for which \(A_\lambda\le3\).
\end{proof}

We also use the following standard dyadic decomposition for weighted
Bergman spaces; see \cite[(2.3)]{Blasco1995} and
\cite{MateljevicPavlovic1984,BKV1999}.

\begin{proposition}
\label{prop:bergman-dyadic-model}
Let \(1<p<\infty\) and \(\beta>-1\). Then, for every
\(g\in\Hol(\D)\),
\[
        \|g\|_{A^p_\beta}^p
        \asymp_{p,\beta}
        \sum_{j=0}^{\infty}
        2^{-j(\beta+1)}\|\Delta_jg\|_{H^p}^p.
\]
\end{proposition}

The interval projection estimate also holds on weighted Bergman
spaces.

\begin{lemma}
\label{lem:interval-projections}
Let \(1<p<\infty\), \(\alpha>-1\), and let
\(J\subset\N_0\) be a finite integer interval. Then
\[
        \|\Pi_Jh\|_{A^p_\alpha}
        \le C_{p,\alpha}\|h\|_{A^p_\alpha},
        \qquad h\in A^p_\alpha,
\]
where \(C_{p,\alpha}\) is independent of \(J\).
\end{lemma}

\begin{proof}
Apply Lemma~\ref{lem:Hp-interval-projections} to
\(h_r(z)=h(rz)\) and integrate in polar coordinates, using
\(\Pi_Jh_r=(\Pi_Jh)_r\).
\end{proof}

We also use the following dyadic Bernstein estimate; see
\cite[Lemmas~4.4.6 and~7.4.1]{JevticVukoticArsenovic2016}.
It also follows from Lemma~\ref{lem:bv-multiplier}, applied to the
multipliers \((n/2^j)\mathbf1_{I_j}(n)\) and
\((2^j/n)\mathbf1_{I_j}(n)\).

\begin{lemma}
\label{lem:dyadic-Bernstein}
Let \(1<p<\infty\), \(j\ge1\), and let \(P\) be an analytic
polynomial with \(\operatorname{supp}\widehat P\subset I_j\). Then
\[
\|P'\|_{H^p}\asymp_p 2^j\|P\|_{H^p}.
\]
\end{lemma}

By the preceding estimates, we obtain the following dyadic
characterization of the Dirichlet-type spaces.

\begin{proposition}
\label{prop:dirichlet-dyadic-model}
Let \(1<p<\infty\) and \(\alpha>-1\). Then, for every
\(f\in\Hol(\D)\),
\[
        \|f\|_{\mathcal D^p_\alpha}^p
        \asymp_{p,\alpha}
        |f(0)|^p
        +
        \sum_{j=0}^{\infty}
        2^{j(p-\alpha-1)}
        \|\Delta_jf\|_{H^p}^p.
\]
\end{proposition}

\begin{proof}
We first consider an analytic polynomial \(P\), and put
\(P_j=\Delta_jP\). For \(j\ge1\), set
\(E_j=\{2^j-1,\ldots,2^{j+1}-2\}\). Then
\[
        \operatorname{supp}\widehat{P_j'}
        \subset E_j
        \subset I_{j-1}\cup I_j.
\]
For \(j,k\ge1\),
\[
\begin{aligned}
\Delta_k(P')
&=
\Pi_{I_k}(P')
=
\Pi_{I_k}(P_k'+P_{k+1}') \\
&=
\Pi_{I_k\cap E_k}(P_k')
+
\Pi_{I_k\cap E_{k+1}}(P_{k+1}'),
\end{aligned}
\]
and
\[
\begin{aligned}
P_j'
&=
\Pi_{E_j}(P')
=
\Pi_{E_j}\bigl(\Delta_{j-1}(P')+\Delta_j(P')\bigr) \\
&=
\Pi_{E_j\cap I_{j-1}}\Delta_{j-1}(P')
+
\Pi_{E_j\cap I_j}\Delta_j(P').
\end{aligned}
\]
Also,
\[
\Delta_0(P')
=
P_0'+\Pi_{\{1\}}(P_1'),
\qquad
P_0'=\widehat P(1).
\]
By Proposition~\ref{prop:bergman-dyadic-model},
\[
\|P'\|_{A^p_\alpha}^p
\asymp_{p,\alpha}
\sum_{k=0}^{\infty}
2^{-k(\alpha+1)}
\|\Delta_k(P')\|_{H^p}^p.
\]
The preceding identities and
Lemma~\ref{lem:Hp-interval-projections}, together with
\(2^{-j(\alpha+1)}
\asymp_\alpha2^{-(j-1)(\alpha+1)}\), give
\[
\sum_{k=0}^{\infty}
2^{-k(\alpha+1)}
\|\Delta_k(P')\|_{H^p}^p
\asymp_{p,\alpha}
|\widehat P(1)|^p
+
\sum_{j=1}^{\infty}
2^{-j(\alpha+1)}
\|(\Delta_jP)'\|_{H^p}^p.
\]
Applying Lemma~\ref{lem:dyadic-Bernstein} now yields
\[
\|P'\|_{A^p_\alpha}^p
\asymp_{p,\alpha}
|\widehat P(1)|^p
+
\sum_{j=1}^{\infty}
2^{j(p-\alpha-1)}
\|\Delta_jP\|_{H^p}^p.
\]
Since
\(|P(0)|^p+|\widehat P(1)|^p\asymp_p\|\Delta_0P\|_{H^p}^p\),
this proves the asserted equivalence for analytic polynomials.

For a general \(f\in\Hol(\D)\), put
\(Q_Nf=\sum_{j=0}^N\Delta_jf\).
Suppose first that the dyadic sum in the statement is finite. Applying
the polynomial estimate to \(Q_Nf-Q_Mf\) shows that
\((Q_Nf)\) is Cauchy in \(\mathcal D^p_\alpha\). Its limit is \(f\):
indeed, \(Q_Nf\to f\) locally uniformly, while point evaluations are
continuous on \(\mathcal D^p_\alpha\). Passing to the limit gives
\[
        \|f\|_{\mathcal D^p_\alpha}^p
        \lesssim_{p,\alpha}
        |f(0)|^p
        +
        \sum_{j=0}^{\infty}
        2^{j(p-\alpha-1)}
        \|\Delta_jf\|_{H^p}^p.
\]

Conversely, suppose that \(f\in\mathcal D^p_\alpha\). Then
\[
        (Q_Nf)'
        =
        \Pi_{\{0,\ldots,2^{N+1}-2\}}(f'),
        \qquad
        Q_Nf(0)=f(0).
\]
Applying Lemma~\ref{lem:Hp-interval-projections} to the radial
dilates of \(f'\) and integrating against the Bergman weight gives
\[
        \|(Q_Nf)'\|_{A^p_\alpha}
        \lesssim_p
        \|f'\|_{A^p_\alpha},
\]
and hence
\[
        \|Q_Nf\|_{\mathcal D^p_\alpha}
        \lesssim_p
        \|f\|_{\mathcal D^p_\alpha},
\]
uniformly in \(N\). Applying the polynomial estimate to \(Q_Nf\)
therefore gives
\[
        |f(0)|^p
        +
        \sum_{j=0}^{N}
        2^{j(p-\alpha-1)}
        \|\Delta_jf\|_{H^p}^p
        \lesssim_{p,\alpha}
        \|f\|_{\mathcal D^p_\alpha}^p.
\]
Letting \(N\to\infty\) proves the reverse inequality.
\end{proof}

\subsection{Analytic Besov and mean-Lipschitz spaces}
\label{subsec:analytic-besov}

For \(s\in\mathbb R\), \(1<q<\infty\), and
\(1\le u\le\infty\), let \(B^s_{q,u}\) be the analytic Besov space
consisting of the functions \(F\in\Hol(\D)\) such that
\[
\|F\|_{B^s_{q,u}}
:=
\left\|
\bigl(2^{js}\|\Delta_jF\|_{H^q}\bigr)_{j\ge0}
\right\|_{\ell^u}
<\infty.
\]

We first compare this norm with the standard Besov norm defined by a
smooth dyadic decomposition; see
\cite[Section~2.1]{Peller2024Besov}. Let
\(w\in C^\infty(\mathbb R)\) satisfy
\[
w\ge0,\qquad
\operatorname{supp}w\subset[1/2,2],\qquad
w(t)=1-w(t/2),\quad 1\le t\le2,
\]
and define
\[
\widetilde\Delta_jF(z)
=
\sum_{n\ge0}w(n/2^j)\widehat F(n)z^n
\quad (j\ge1),
\qquad
\widetilde\Delta_0F(z)
=
\widehat F(0)+\widehat F(1)z.
\]
The sequences
\(\bigl(w(n/2^j)\bigr)_{n\ge0}\), \(j\ge1\), have uniformly bounded
variation, since
\[
\sup_{n\ge0}\left|w\!\left(\frac n{2^j}\right)\right|
\le \|w\|_\infty,
\qquad
\sum_{n\ge0}
\left|
w\!\left(\frac{n+1}{2^j}\right)
-
w\!\left(\frac n{2^j}\right)
\right|
\le
\|w'\|_{L^1},
\]
uniformly in \(j\). Thus Lemma~\ref{lem:bv-multiplier} shows that
the operators \(\widetilde\Delta_j\) are uniformly bounded on \(H^q\).

Since \(\operatorname{supp}w\subset[1/2,2]\) and
\(w(1/2)=w(2)=0\), we have \(w(n/2^j)=0\) unless
\(2^{j-1}<n<2^{j+1}\). Hence
\(\widetilde\Delta_jF\) involves only the Taylor coefficients with
indices in \(I_{j-1}\cup I_j\).
On the other hand, if \(n\in I_j\), then \(1\le n/2^j<2\), and hence
\(
w(n/2^j)+w(n/2^{j+1})=1.
\) Hence, for \(j\ge1\),
\[
\widetilde\Delta_jF
=
\widetilde\Delta_j(\Delta_{j-1}F+\Delta_jF),
\qquad\text{and}\qquad
\Delta_jF
=
\Delta_j(\widetilde\Delta_jF+\widetilde\Delta_{j+1}F).
\]
By Lemmas~\ref{lem:bv-multiplier} and
\ref{lem:Hp-interval-projections}, respectively,
\[
\begin{aligned}
\|\widetilde\Delta_jF\|_{H^q}
&\lesssim_q
\|\Delta_{j-1}F\|_{H^q}+\|\Delta_jF\|_{H^q},\\
\|\Delta_jF\|_{H^q}
&\lesssim_q
\|\widetilde\Delta_jF\|_{H^q}
+\|\widetilde\Delta_{j+1}F\|_{H^q},
\qquad j\ge1.
\end{aligned}
\]
Since
\(2^{(j-1)s}\asymp_s2^{js}\asymp_s2^{(j+1)s}\),
these estimates, together with
\(\widetilde\Delta_0F=\Delta_0F\), give
\[
\left\|
\bigl(2^{js}\|\Delta_jF\|_{H^q}\bigr)_{j\ge0}
\right\|_{\ell^u}
\asymp
\left\|
\bigl(2^{js}\|\widetilde\Delta_jF\|_{H^q}\bigr)_{j\ge0}
\right\|_{\ell^u}.
\]
Hence the norm defined by \((\Delta_j)\) is equivalent to the standard
smooth Besov norm.

By Proposition~\ref{prop:dirichlet-dyadic-model},
\(\mathcal D^p_\alpha=B^{\,1-(\alpha+1)/p}_{p,p}\) and
\(B^p=B^{1/p}_{p,p}\), with equivalent norms, for
\(1<p<\infty\) and \(\alpha>-1\).
We denote by \(b^s_{q,\infty}\) the closure of the analytic
polynomials in \(B^s_{q,\infty}\).

\begin{lemma}
\label{lem:little-besov-distance}
Let \(s\in\mathbb R\), \(1<q<\infty\), and
\(F\in B^s_{q,\infty}\). Then
\[
\operatorname{dist}_{B^s_{q,\infty}}
\bigl(F,b^s_{q,\infty}\bigr)
=
\limsup_{j\to\infty}
2^{js}\|\Delta_jF\|_{H^q}.
\]
In particular,
\[
F\in b^s_{q,\infty}
\quad\Longleftrightarrow\quad
2^{js}\|\Delta_jF\|_{H^q}\longrightarrow0.
\]
\end{lemma}

\begin{proof}
Let \(F_N=\sum_{j=0}^N\Delta_jF\). Then
\(F_N\in b^s_{q,\infty}\) and
\[
\|F-F_N\|_{B^s_{q,\infty}}
=
\sup_{j>N}2^{js}\|\Delta_jF\|_{H^q}.
\]
This gives
\[
\operatorname{dist}_{B^s_{q,\infty}}
\bigl(F,b^s_{q,\infty}\bigr)
\le
\limsup_{j\to\infty}2^{js}\|\Delta_jF\|_{H^q}.
\]

Conversely, let \(G\in b^s_{q,\infty}\). Given
\(\varepsilon>0\), choose an analytic polynomial \(P\) such that
\(\|G-P\|_{B^s_{q,\infty}}<\varepsilon\). For all sufficiently large
\(j\), \(\Delta_jP=0\), and hence
\[
2^{js}\|\Delta_jG\|_{H^q}
\le
\|G-P\|_{B^s_{q,\infty}}
<\varepsilon.
\]
Thus \(2^{js}\|\Delta_jG\|_{H^q}\to0\), and therefore
\[
\limsup_{j\to\infty}
2^{js}\|\Delta_jF\|_{H^q}
\le
\|F-G\|_{B^s_{q,\infty}}.
\]
Taking the infimum over \(G\in b^s_{q,\infty}\) gives the reverse
inequality. The last assertion follows immediately.
\end{proof}

We shall also use the radial characterization of \(B^s_{q,u}\);
see \cite[Section~2.1, formulas~(2.5)--(2.6)]{Peller2024Besov}.
Write \(M_q(\rho,F)=\|F(\rho,\cdot)\|_{L^q(\T)}\), and let
\(m\ge0\) be an integer with \(m>s\). For \(1\le u<\infty\),
\[
\|F\|_{B^s_{q,u}}
\asymp_{s,q,u,m}
\sum_{k=0}^{m-1}|F^{(k)}(0)|
+
\left(
\int_0^1
M_q(\rho,F^{(m)})^u
(1-\rho)^{(m-s)u-1}\,d\rho
\right)^{1/u}.
\]
For \(u=\infty\), the integral term is replaced by
\(\sup_{0<\rho<1}(1-\rho)^{m-s}M_q(\rho,F^{(m)})\).
When \(m=0\), the sum is omitted and \(F^{(0)}=F\). Moreover,
\[
F\in b^s_{q,\infty}
\quad\Longleftrightarrow\quad
(1-\rho)^{m-s}M_q(\rho,F^{(m)})
\longrightarrow0
\qquad (\rho\to1^-).
\]
For \(s<0\), taking \(m=0\) in the radial characterization above
identifies \(B^s_{q,u}\) with the corresponding mixed norm space. In
particular, \(B^s_{q,q}=A^q_{-sq-1}\), with equivalent norms.

For \(1<q<\infty\) and \(0<s<1\), let \(\Lambda_s^q\) be the
mean-Lipschitz space consisting of the functions \(F\in\Hol(\D)\)
such that
\[
\|F\|_{\Lambda_s^q}
:=
|F(0)|
+
\sup_{0<\rho<1}
(1-\rho)^{1-s}M_q(\rho,F')
<\infty,
\]
and let \(\lambda_s^q\) consist of those
\(F\in\Lambda_s^q\) for which
\[
\lim_{\rho\to1^-}(1-\rho)^{1-s}M_q(\rho,F')
=0.
\]
Taking \(m=1\) in the radial characterization above gives
\(B^s_{q,\infty}=\Lambda_s^q\) and
\(b^s_{q,\infty}=\lambda_s^q\), with equivalent norms; see also
\cite[Theorem~A and Remark~1]{GalanopoulosGirela2026}.

\begin{proposition}
\label{prop:dyadic-mean-lipschitz}
Let \(1<q<\infty\), \(0<s<1\), and \(F\in\Hol(\D)\). Then
\[
\begin{aligned}
F\in\Lambda_s^q
&\quad\Longleftrightarrow\quad
\sup_{j\ge0}2^{js}\|\Delta_jF\|_{H^q}<\infty,\\
F\in\lambda_s^q
&\quad\Longleftrightarrow\quad
\lim_{j\to\infty}2^{js}\|\Delta_jF\|_{H^q}=0.
\end{aligned}
\]
\end{proposition}

\begin{proof}
The first assertion follows from
\(B^s_{q,\infty}=\Lambda_s^q\) and the definition of
\(B^s_{q,\infty}\). The second follows from
\(b^s_{q,\infty}=\lambda_s^q\) and
Lemma~\ref{lem:little-besov-distance}.
\end{proof}
\subsection{Analytic truncated Besov spaces}
\label{subsec:analytic-truncated-besov}

We introduce the following analytic dyadic realization of the truncated
Besov spaces introduced by Dom\'inguez and Tikhonov
\cite[Definition~3.1(i)]{DominguezTikhonov2024}.

\begin{definition}
\label{def:analytic-truncated-besov}
Let \(s\in\mathbb R\), \(b>0\), \(1<q<\infty\), and
\(q\le u\le\infty\). The analytic truncated Besov space
\(T_u^bB^s_{q,q}\) consists of all \(F\in\Hol(\D)\) such that
\[
        \|F\|_{T_u^bB^s_{q,q}}
        :=
        \left\|
        \left(
        2^{kb}
        \left(
        \sum_{j=2^k-1}^{2^{k+1}-2}
        2^{jsq}\|\Delta_jF\|_{H^q}^q
        \right)^{1/q}
        \right)_{k\ge0}
        \right\|_{\ell^u}
        <\infty.
\]
The little analytic truncated Besov space
\(t_u^bB^s_{q,q}\) is the closure of the analytic polynomials in
\(T_u^bB^s_{q,q}\).
\end{definition}

\begin{proposition}
\label{prop:truncated-besov-tail}
Let \(s\in\mathbb R\), \(b>0\), \(1<q<\infty\),
\(q\le u\le\infty\), and \(F\in\Hol(\D)\). Set
\[
        a_j=2^{js}\|\Delta_jF\|_{H^q},
        \qquad
        H_J=\left(\sum_{j\ge J}a_j^q\right)^{1/q}.
\]
Then the following assertions hold.

\begin{enumerate}[label=\textup{(\roman*)},leftmargin=2.2em]
\item If \(q\le u<\infty\), then
\[
        \|F\|_{T_u^bB^s_{q,q}}
        \asymp_{b,q,u}
        \left(
        \sum_{J\ge0}(J+1)^{bu-1}H_J^u
        \right)^{1/u}.
\]
Moreover,
\[
        t_u^bB^s_{q,q}=T_u^bB^s_{q,q}.
\]

\item If \(u=\infty\), then
\[
        \|F\|_{T_\infty^bB^s_{q,q}}
        \asymp_{b,q}
        \sup_{J\ge0}(J+1)^bH_J,
\]
and
\[
        t_\infty^bB^s_{q,q}
        =
        \left\{
        F\in T_\infty^bB^s_{q,q}:
        \lim_{J\to\infty}(J+1)^bH_J=0
        \right\}.
\]
If \(F\in T_\infty^bB^s_{q,q}\), then
\[
\operatorname{dist}_{T_\infty^bB^s_{q,q}}
\bigl(F,t_\infty^bB^s_{q,q}\bigr)
=
\limsup_{k\to\infty}
2^{kb}
\left(
\sum_{j=2^k-1}^{2^{k+1}-2}a_j^q
\right)^{1/q} 
\asymp_{b,q}
\limsup_{J\to\infty}(J+1)^bH_J.
\]
\end{enumerate}
\end{proposition}

\begin{proof}
Since \((H_J)\) is nonincreasing, dyadic grouping gives
\begin{equation}
\label{eq:truncated-condensation-sum}
\sum_{J\ge0}(J+1)^{bu-1}H_J^u
\asymp_{b,u}
\sum_{k\ge0}2^{kbu}H_{2^k-1}^u,
\qquad q\le u<\infty,
\end{equation}
and
\begin{equation}
\label{eq:truncated-condensation-sup}
\sup_{J\ge0}(J+1)^bH_J
\asymp_b
\sup_{k\ge0}2^{kb}H_{2^k-1}.
\end{equation}
Put
\(c_k=2^{kb}\left(\sum_{j=2^k-1}^{2^{k+1}-2}a_j^q\right)^{1/q}\).
Since the intervals
\(\{2^k-1,\ldots,2^{k+1}-2\}\), \(k\ge0\), partition
\(\mathbb N_0\),
\begin{equation}
\label{eq:truncated-tail-identity}
\bigl(2^{kb}H_{2^k-1}\bigr)^q
=
\sum_{\ell\ge k}
2^{-(\ell-k)bq}c_\ell^q.
\end{equation}

Suppose first that \(u<\infty\). Since
\((2^{-mbq})_{m\ge0}\in\ell^1\) and \(u/q\ge1\), Young's
inequality on \(\ell^{u/q}\), applied to
\eqref{eq:truncated-tail-identity}, together with the term
\(\ell=k\), gives
\[
\|(c_k)\|_{\ell^u}
\le
\bigl\|(2^{kb}H_{2^k-1})\bigr\|_{\ell^u}
\le
(1-2^{-bq})^{-1/q}\|(c_k)\|_{\ell^u}.
\]
Together with \eqref{eq:truncated-condensation-sum}, this proves the
norm equivalence in \textup{(i)}.

For \(u=\infty\), \eqref{eq:truncated-tail-identity} gives
\begin{equation}
\label{eq:truncated-sup-tail}
c_k
\le
2^{kb}H_{2^k-1}
\le
(1-2^{-bq})^{-1/q}\sup_{\ell\ge k}c_\ell.
\end{equation}
Together with \eqref{eq:truncated-condensation-sup}, this proves the
norm equivalence in \textup{(ii)} and shows that
\[
\lim_{J\to\infty}(J+1)^bH_J=0
\quad\Longleftrightarrow\quad
\lim_{k\to\infty}c_k=0.
\]

For \(K\ge1\), set
\(P_KF=\sum_{j=0}^{2^K-2}\Delta_jF\).
Then \(P_KF\) is an analytic polynomial and
\begin{equation}
\label{eq:truncated-polynomial-tail}
\|F-P_KF\|_{T_u^bB^s_{q,q}}
=
\|(c_k)_{k\ge K}\|_{\ell^u}.
\end{equation}
Hence analytic polynomials are dense in \(T_u^bB^s_{q,q}\) when
\(u<\infty\). If \(u=\infty\) and \(c_k\to0\),
\eqref{eq:truncated-polynomial-tail} gives
\(P_KF\to F\) in \(T_\infty^bB^s_{q,q}\).

Conversely, suppose that \(F\in t_\infty^bB^s_{q,q}\). Given
\(\varepsilon>0\), choose an analytic polynomial \(P\) such that
\(\|F-P\|_{T_\infty^bB^s_{q,q}}<\varepsilon\).
For all sufficiently large \(k\), we have \(c_k(P)=0\), and hence
\[
c_k=c_k(F-P)
\le
\|F-P\|_{T_\infty^bB^s_{q,q}}
<\varepsilon.
\]
Thus \(c_k\to0\), proving the characterization of
\(t_\infty^bB^s_{q,q}\).

It remains to prove the distance formula. Since
\(P_KF\in t_\infty^bB^s_{q,q}\), \eqref{eq:truncated-polynomial-tail}
gives
\[
\begin{aligned}
\operatorname{dist}_{T_\infty^bB^s_{q,q}}
\bigl(F,t_\infty^bB^s_{q,q}\bigr)
&\le
\lim_{K\to\infty}
\|F-P_KF\|_{T_\infty^bB^s_{q,q}} \\
&=
\lim_{K\to\infty}\sup_{k\ge K}c_k
=
\limsup_{k\to\infty}c_k.
\end{aligned}
\]
Conversely, let \(G\in t_\infty^bB^s_{q,q}\), and denote its block
quantities by \(c_k(G)\). Since \(c_k(G)\to0\),
\[
c_k
\le
c_k(F-G)+c_k(G)
\le
\|F-G\|_{T_\infty^bB^s_{q,q}}+c_k(G).
\]
Therefore
\[
\limsup_{k\to\infty}c_k
\le
\|F-G\|_{T_\infty^bB^s_{q,q}}.
\]
Taking the infimum over \(G\in t_\infty^bB^s_{q,q}\) gives
\begin{equation}
\label{eq:truncated-distance-ck}
\operatorname{dist}_{T_\infty^bB^s_{q,q}}
\bigl(F,t_\infty^bB^s_{q,q}\bigr)
=
\limsup_{k\to\infty}c_k.
\end{equation}

Finally, the same dyadic grouping as in
\eqref{eq:truncated-condensation-sup} gives, for every \(K\ge0\),
\[
\sup_{J\ge2^K-1}(J+1)^bH_J
\asymp_b
\sup_{k\ge K}2^{kb}H_{2^k-1}.
\]
Hence
\[
\limsup_{J\to\infty}(J+1)^bH_J
\asymp_b
\limsup_{k\to\infty}2^{kb}H_{2^k-1}.
\]
By \eqref{eq:truncated-sup-tail},
\[
\limsup_{k\to\infty}c_k
\asymp_{b,q}
\limsup_{k\to\infty}2^{kb}H_{2^k-1}.
\]
Combining this with \eqref{eq:truncated-distance-ck} proves the
distance formula in \textup{(ii)} and completes the proof.
\end{proof}

\begin{corollary}
\label{cor:critical-truncated-besov-norms}
Let \(1<p,q<\infty\), \(\beta>-1\), and put
\(s=1-(\beta+1)/q\). Then the following assertions hold.

\begin{enumerate}[label=\textup{(\roman*)},leftmargin=2.2em]
\item If \(q<p\) and \(1/r=1/q-1/p\), then
\[
        G_{p,q,\beta}(\eta)^{1/q}
        \asymp_{p,q}
        \|F_\eta\|_{T_r^{1/p'}B^s_{q,q}}.
\]

\item If \(p\le q\), then
\[
        \sup_{J\ge0}(J+1)^{1/p'}W_J(\eta)^{1/q}
        \asymp_{p,q}
        \|F_\eta\|_{T_\infty^{1/p'}B^s_{q,q}}.
\]
Moreover,
\[
        \lim_{J\to\infty}(J+1)^{1/p'}W_J(\eta)^{1/q}=0
        \quad\Longleftrightarrow\quad
        F_\eta\in t_\infty^{1/p'}B^s_{q,q}.
\]
If \(F_\eta\in T_\infty^{1/p'}B^s_{q,q}\), then
\[
\begin{aligned}
\operatorname{dist}_{T_\infty^{1/p'}B^s_{q,q}}
\bigl(F_\eta,t_\infty^{1/p'}B^s_{q,q}\bigr)
&=
\limsup_{k\to\infty}
2^{k/p'}
\left(
\sum_{j=2^k-1}^{2^{k+1}-2}
2^{jsq}\|\Delta_jF_\eta\|_{H^q}^q
\right)^{1/q} \\
&\asymp_{p,q}
\limsup_{J\to\infty}
(J+1)^{1/p'}W_J(\eta)^{1/q}.
\end{aligned}
\]
\end{enumerate}
\end{corollary}

\begin{proof}
For \(F_\eta\), the tails in
Proposition~\ref{prop:truncated-besov-tail} satisfy
\(H_J=W_J(\eta)^{1/q}\). If \(q<p\) and \(1/r=1/q-1/p\), then
\(r>q\) and \(r/p'-1=r/q'\). Thus \textup{(i)} follows from
Proposition~\ref{prop:truncated-besov-tail} with \(b=1/p'\) and
\(u=r\). Taking \(b=1/p'\) and \(u=\infty\) gives
\textup{(ii)}, the characterization of
\(t_\infty^{1/p'}B^s_{q,q}\), and the distance formula.
\end{proof}

\subsection{Discrete Hardy inequalities and local triangular estimates}

For \(N\ge1\), let
\(D_N(e^{it})=\sum_{n=0}^{N-1}e^{int}\) and
\(K_N(t)=\min\{N,|t|^{-1}\}\) on \((-\pi,\pi]\), with \(K_N(0)=N\).
For \(t\ne0\),
\(D_N(e^{it})=e^{i(N-1)t/2}\sin(Nt/2)/\sin(t/2)\), and hence
\[
        |D_N(e^{it})|
        \lesssim
        \min\{N,|t|^{-1}\}
        =
        K_N(t).
\]
Moreover,
\(\|K_N\|_{L^{p'}(\T)}\lesssim_p N^{1/p}\), and therefore
\(\|D_N\|_{L^{p'}(\T)}\lesssim_p N^{1/p}\).

\begin{lemma}
\label{lem:block-summation-estimate}
Let \(1<p<\infty\), let \(I\subset\mathbb N_0\) be a finite interval,
and let \(u(z)=\sum_{n\in I}c_nz^n\).  Then
\[
        \left|\sum_{n\in I}c_n\right|
        \le
        C_p |I|^{1/p}\|u\|_{H^p}.
\]
In particular,
\(\left|\sum_{n\in I_j}c_n\right|\le C_p2^{j/p}\|u\|_{H^p}\).
\end{lemma}

\begin{proof}
Write \(I=\{M,\ldots,M+N-1\}\), where \(N=|I|\), and factor out
\(z^M\). For \(I=\{0,\ldots,N-1\}\),
\[
        \sum_{n=0}^{N-1}c_n
        =
        \int_{\T}u\overline{D_N}\,dm.
\]
H\"older's inequality and
\(\|D_N\|_{L^{p'}}\lesssim_pN^{1/p}\) give the result.
\end{proof}

\begin{lemma}
\label{lem:normalized-interval-averages}
Let \(1<p<\infty\), let \(I\subset\mathbb N_0\) be a finite interval,
and put \(N=|I|\). If
\(u_I=N^{-1}\sum_{k\in I}z^k\), then
\[
        \|u_I\|_{H^p}^p\le C_pN^{-1}.
\]
\end{lemma}

\begin{proof}
Writing \(I=\{M,\ldots,M+N-1\}\), we have
\(u_I=z^MN^{-1}D_N\). The standard Dirichlet kernel estimate
\(\|D_N\|_{H^p}\lesssim_pN^{1-1/p}\) gives
\(\|u_I\|_{H^p}\lesssim_pN^{-1/p}\).
\end{proof}

Let \(I\subset\mathbb N_0\) be a finite interval. For analytic
polynomials
\(u(z)=\sum_{k\in I}a_kz^k\) and
\(v(z)=\sum_{n\in I}b_nz^n\), define
\[
        \mathcal T_I(u,v)(z)
        =
        \sum_{n\in I}
        \left(
        \sum_{\substack{k\in I\\ k\le n}}a_k
        \right)b_nz^n.
\]

\begin{lemma}
\label{lem:local-triangular-estimate}
Let \(1<p,q<\infty\). Then, for every finite interval
\(I\subset\mathbb N_0\) and all analytic polynomials \(u,v\) whose
Taylor supports are contained in \(I\),
\[
\|\mathcal T_I(u,v)\|_{H^q}
\le
C_{p,q}|I|^{1/p}\|u\|_{H^p}\|v\|_{H^q}.
\]
\end{lemma}

\begin{proof}
Let \(I=\{L,\ldots,L+N-1\}\). By factoring out \(z^L\) from
\(u\) and \(v\), and using that multiplication by a monomial is
isometric on every \(H^r\), it suffices to consider
\(I=\{0,\ldots,N-1\}\). Write
\[
A_n=\sum_{k=0}^n a_k,
\qquad
D_{n+1}(e^{-is})=\sum_{k=0}^n e^{-iks}.
\]
Orthogonality gives
\[
A_n
=
\int_{\mathbb T}
u(e^{is})D_{n+1}(e^{-is})\,dm(s).
\]
Using this representation and interchanging the finite sum with the
integral, we obtain
\[
\begin{aligned}
\mathcal T_I(u,v)(e^{it})
&=
\sum_{n=0}^{N-1}A_nb_ne^{int} \\
&=
\int_{\mathbb T}u(e^{is})
\left(
\sum_{n=0}^{N-1}
D_{n+1}(e^{-is})b_ne^{int}
\right)\,dm(s).
\end{aligned}
\]
Thus
\[
\mathcal T_I(u,v)(e^{it})
=
\int_{\mathbb T}u(e^{is})V_sv(e^{it})\,dm(s),
\]
where
\[
V_sv(e^{it})
=
\sum_{n=0}^{N-1}D_{n+1}(e^{-is})b_ne^{int}.
\]
We claim that
\[
\|V_sv\|_{H^q}
\le
C_qK_N(s)\|v\|_{H^q}.
\]
Indeed, for \(0<|s|\le\pi\), since
\[
D_{n+1}(e^{-is})
=
\frac{1-e^{-i(n+1)s}}{1-e^{-is}},
\]
we have
\[
V_sv(e^{it})
=
\frac{v(e^{it})-e^{-is}v(e^{i(t-s)})}
     {1-e^{-is}}.
\]
Hence, since
\[
|1-e^{-is}|=2|\sin(s/2)|\gtrsim |s|,
\]
it follows that
\[
\|V_sv\|_{H^q}
\lesssim
|s|^{-1}\|v\|_{H^q}.
\]
On the other hand, \(V_s\) is the coefficient multiplier associated
with
\[
\lambda_n^{(s,N)}
=
\begin{cases}
D_{n+1}(e^{-is}),&0\le n\le N-1,\\
0,&n\ge N.
\end{cases}
\]
Since
\[
\sup_n|\lambda_n^{(s,N)}|\le N,
\qquad
|\lambda_{n+1}^{(s,N)}-\lambda_n^{(s,N)}|=1
\quad (0\le n\le N-2),
\]
while
\[
|\lambda_N^{(s,N)}-\lambda_{N-1}^{(s,N)}|
=
|D_N(e^{-is})|
\le N,
\]
it follows that
\[
\sup_n|\lambda_n^{(s,N)}|
+
\sum_{n=0}^{\infty}
|\lambda_{n+1}^{(s,N)}-\lambda_n^{(s,N)}|
\lesssim N.
\]
Lemma~\ref{lem:bv-multiplier}, applied with exponent \(q\), gives
\[
\|V_sv\|_{H^q}\le C_qN\|v\|_{H^q}.
\]
The claim follows.

By Minkowski's integral inequality, H\"older's inequality, and the
estimate for \(K_N\),
\[
\begin{aligned}
\|\mathcal T_I(u,v)\|_{H^q}
&\le
\int_{\mathbb T}|u(e^{is})|\|V_sv\|_{H^q}\,dm(s) \\
&\le
C_q\|v\|_{H^q}
\int_{\mathbb T}|u(e^{is})|K_N(s)\,dm(s) \\
&\le
C_q\|u\|_{H^p}\|v\|_{H^q}
\|K_N\|_{L^{p'}(\mathbb T)} \\
&\lesssim_{p,q}
N^{1/p}\|u\|_{H^p}\|v\|_{H^q}.
\end{aligned}
\]
Since \(N=|I|\), the proof is complete.
\end{proof}

We shall use the following form of the discrete weighted Hardy inequality
for the Hardy operator \(Hx(j)=\sum_{m=0}^{j}x_m\). It is a special
case of \cite[Theorem~1.1]{OkpotiPerssonWedestig2006}, after an index
shift. We include a proof for convenience.

\begin{lemma}
\label{lem:tail-Hardy}
Let \(1<p<\infty\), let \(w_j\ge0\), and suppose that
\[
        A
        :=
        \sup_{J\ge0}
        (J+1)^{p-1}\sum_{j\ge J}w_j
        <\infty.
\]
Then, for every nonnegative sequence \(x=(x_j)_{j\ge0}\),
\[
        \sum_{j\ge0}
        w_j\left(\sum_{m\le j}x_m\right)^p
        \le
        C_p A\sum_{m\ge0}x_m^p.
\]
\end{lemma}

\begin{proof}
We first assume that \(x\) is finitely supported. Put
\(S_j=\sum_{m=0}^j x_m\) and \(W_J=\sum_{j\ge J}w_j\). Then
\(W_J\le A(J+1)^{1-p}\) and \(w_j=W_j-W_{j+1}\).

For \(L\ge0\), Abel summation gives
\[
        \sum_{j=0}^{L}w_jS_j^p
        =
        \sum_{j=0}^{L}W_j(S_j^p-S_{j-1}^p)
        -
        W_{L+1}S_L^p,
        \qquad S_{-1}=0.
\]
Since \((S_L)\) is bounded and \(W_{L+1}\to0\), letting
\(L\to\infty\) gives
\[
        \sum_{j\ge0}w_jS_j^p
        =
        \sum_{j\ge0}W_j(S_j^p-S_{j-1}^p).
\]
Since \(0\le S_{j-1}\le S_j\),
\[
        S_j^p-S_{j-1}^p
        \le
        p(S_j-S_{j-1})S_j^{p-1}
        =
        px_jS_j^{p-1}.
\]
Hence
\begin{equation}
\label{eq:tail-Hardy-main}
        \sum_{j\ge0}w_jS_j^p
        \le
        C_pA
        \sum_{j\ge0}
        x_j\left(\frac{S_j}{j+1}\right)^{p-1}.
\end{equation}
By Hölder's inequality,
\[
        \sum_{j\ge0}
        x_j\left(\frac{S_j}{j+1}\right)^{p-1}
        \le
        \left(\sum_{j\ge0}x_j^p\right)^{1/p}
        \left(
        \sum_{j\ge0}
        \left(\frac{S_j}{j+1}\right)^p
        \right)^{1/p'}.
\]
The classical discrete Hardy inequality gives
\begin{equation}
\label{eq:tail-Hardy-discrete}
        \sum_{j\ge0}
        \left(\frac{S_j}{j+1}\right)^p
        \le
        C_p\sum_{j\ge0}x_j^p.
\end{equation}
Combining \eqref{eq:tail-Hardy-main} and
\eqref{eq:tail-Hardy-discrete} proves the result for finitely
supported sequences.

For a general nonnegative sequence \(x\), apply the estimate to
\(x^{(N)}=(x_0,\ldots,x_N,0,\ldots)\) and let \(N\to\infty\).
The result follows by monotone convergence.
\end{proof}

\begin{lemma}
\label{lem:mixed-tail-Hardy}
Let \(1<p\le q<\infty\), let \(w_j\ge0\), and suppose that
\[
        A
        :=
        \sup_{J\ge0}
        (J+1)^{q/p'}\sum_{j\ge J}w_j
        <\infty.
\]
Then, for every nonnegative sequence \(x=(x_j)_{j\ge0}\),
\[
        \sum_{j\ge0}
        w_j\left(\sum_{m\le j}x_m\right)^q
        \le
        C_{p,q}A
        \left(\sum_{m\ge0}x_m^p\right)^{q/p}.
\]
\end{lemma}

\begin{proof}
If \(q=p\), the result follows from
Lemma~\ref{lem:tail-Hardy}. Hence assume that \(q>p\).
Put \(S_j=\sum_{m=0}^j x_m\) and
\(c=(q-p)/p'>0\). Hölder's inequality gives
\[
        S_j^{q-p}
        \le
        (j+1)^c
        \left(\sum_{m\ge0}x_m^p\right)^{(q-p)/p}.
\]
Hence, with \(\widetilde w_j=(j+1)^cw_j\),
\begin{equation}
\label{eq:mixed-tail-reduction}
        \sum_{j\ge0}w_jS_j^q
        \le
        \left(\sum_{m\ge0}x_m^p\right)^{(q-p)/p}
        \sum_{j\ge0}\widetilde w_jS_j^p.
\end{equation}

We verify that \((\widetilde w_j)\) satisfies the hypothesis of
Lemma~\ref{lem:tail-Hardy}. Put
\(W_J=\sum_{j\ge J}w_j\). Then
\(W_J\le A(J+1)^{-q/p'}\). Since
\[
        (L+1)^cW_{L+1}
        \lesssim A(L+1)^{1-p}\longrightarrow0
        \qquad (L\to\infty).
\]
Abel summation gives
\begin{equation}
\label{eq:mixed-tail-weight}
\begin{aligned}
        \sum_{j\ge J}(j+1)^cw_j
        &=
        (J+1)^cW_J
        +
        \sum_{j>J}\bigl((j+1)^c-j^c\bigr)W_j \\
        &\lesssim_{p,q}
        A(J+1)^{1-p}
        +
        A\sum_{j>J}(j+1)^{-p} \\
        &\lesssim_{p,q}
        A(J+1)^{1-p}.
\end{aligned}
\end{equation}
Here we used
\((j+1)^c-j^c\lesssim_{p,q}(j+1)^{c-1}\),
\(c-q/p'=1-p\), and
\(c-1-q/p'=-p\).

By \eqref{eq:mixed-tail-weight} and
Lemma~\ref{lem:tail-Hardy},
\begin{equation}
\label{eq:mixed-tail-Hardy-application}
        \sum_{j\ge0}\widetilde w_jS_j^p
        \le
        C_{p,q}A\sum_{m\ge0}x_m^p.
\end{equation}
Combining \eqref{eq:mixed-tail-reduction} and
\eqref{eq:mixed-tail-Hardy-application} proves the result.
\end{proof}

For \(1<q<p<\infty\), we use the following integrated discrete
weighted Hardy criterion; see Bennett~\cite{Bennett1991}.

\begin{lemma}
\label{lem:integrated-tail-Hardy}
Let \(1<q<p<\infty\), let \(r>0\) satisfy
\(1/r=1/q-1/p\), and let \(q'\) be the conjugate exponent of \(q\).
Let \(w_j\ge0\), put \(W_J=\sum_{j\ge J}w_j\), and define
\[
        B
        =
        \left(
        \sum_{J\ge0}
        (J+1)^{r/q'}W_J^{r/q}
        \right)^{1/r}.
\]
Then
\[
        \left(
        \sum_{j\ge0}w_j
        \left(\sum_{m\le j}x_m\right)^q
        \right)^{1/q}
        \le
        C\left(\sum_{m\ge0}x_m^p\right)^{1/p}
\]
for every nonnegative sequence \((x_m)\) if and only if \(B<\infty\).
Moreover, the best constant \(C\) satisfies
\[
        C\asymp_{p,q}B.
\]
\end{lemma}

\begin{proof}
By the standard \(q<p\) discrete Hardy criterion
\cite[Theorem~7(iii)]{KufnerMaligrandaPersson2007}, the inequality
holds if and only if
\[
        B_*^r
        :=
        \sum_{J\ge0}
        w_JW_J^{r/p}(J+1)^{r/p'}
        <\infty,
\]
and its best constant is comparable to \(B_*\)
\cite[Theorems~2.4 and~3.9]{Sinnamon2022}.
It remains to show that \(B_*\asymp_{p,q}B\).

For \(N\ge0\), set
\(w_j^{(N)}=w_j\mathbf 1_{\{j\le N\}}\).
By monotone convergence,
\(B_*^{(N)}\uparrow B_*\) and \(B^{(N)}\uparrow B\), so it suffices
to consider finitely supported \(w\). Put
\[
        s=\frac rq=1+\frac rp,
        \qquad
        a_J=(J+1)^{r/p'}.
\]
Then \(s-1=r/p\). Since \(s>1\),
\(0\le W_{J+1}\le W_J\), and \(w_J=W_J-W_{J+1}\),
\[
        W_J^s-W_{J+1}^s
        \asymp_s
        w_JW_J^{s-1},
\]
because
\((1-t^s)/(1-t)\asymp_s1\) for \(0\le t<1\).
Summation by parts therefore gives
\begin{equation}
\label{eq:integrated-tail-summation}
\begin{aligned}
        B_*^r
        &\asymp_{p,q}
        \sum_{J\ge0}
        a_J(W_J^s-W_{J+1}^s) \\
        &=
        a_0W_0^s
        +
        \sum_{J\ge1}(a_J-a_{J-1})W_J^s.
\end{aligned}
\end{equation}
Since
\[
        a_J-a_{J-1}
        \asymp_{p,q}
        (J+1)^{r/p'-1}
        =
        (J+1)^{r/q'},
\]
where \(r/p'-1=r/q'\), and \(s=r/q\),
\eqref{eq:integrated-tail-summation} yields
\begin{equation}
\label{eq:integrated-tail-comparison}
        B_*^r
        \asymp_{p,q}
        \sum_{J\ge0}
        (J+1)^{r/q'}W_J^{r/q}
        =
        B^r.
\end{equation}
Thus \(B_*\asymp_{p,q}B\). The proof is complete.
\end{proof}

\section{The case
\texorpdfstring{\(-1<\alpha<p-2\)}{-1 < alpha < p - 2}}
\label{sec:below-critical}

Throughout this section, \(1<p,q<\infty\), \(\beta>-1\), and
\(-1<\alpha<p-2\). For \(m\ge0\), put
\(I_m^+=\{2^m,\ldots,2^{m+1}-1\}\) and
\(\Delta_m^+f(z)=\sum_{n\in I_m^+}\widehat f(n)z^n\).
Thus \(I_m^+=I_m\) for \(m\ge1\), whereas \(I_0^+=\{1\}\).

\begin{lemma}
\label{lem:prefix-estimates}
Let \(f(z)=\sum_{k\ge0}a_kz^k\in\mathcal D^p_\alpha\), and put
\(X_m=2^{m(1-(\alpha+1)/p)}\|\Delta_m^+f\|_{H^p}\). Then
\[
        \sum_{m\ge0}X_m^p
        \lesssim_{p,\alpha}
        \|f\|_{\mathcal D^p_\alpha}^p,
        \qquad
        \sup_{j\ge0}
        \left|\sum_{k=0}^j a_k\right|
        \lesssim_{p,\alpha}
        \|f\|_{\mathcal D^p_\alpha}.
\]
\end{lemma}

\begin{proof}
By Proposition~\ref{prop:dirichlet-dyadic-model}, separating the
constant term,
\[
        |a_0|^p+
        \sum_{m\ge0}
        2^{m(p-\alpha-1)}
        \|\Delta_m^+f\|_{H^p}^p
        \lesssim_{p,\alpha}
        \|f\|_{\mathcal D^p_\alpha}^p.
\]
Since \(p-\alpha-1=p(1-(\alpha+1)/p)\), this proves the first
estimate.

Put \(\sigma=1-(\alpha+2)/p>0\). For every interval
\(J\subset I_m^+\), we have
\(\Pi_J(\Delta_m^+f)=\sum_{n\in J}a_nz^n\). By
Lemmas~\ref{lem:block-summation-estimate} and
\ref{lem:Hp-interval-projections}, together with
\(|J|\le |I_m^+|=2^m\),
\[
\begin{aligned}
        \left|\sum_{n\in J}a_n\right|
        &\lesssim_p
        |J|^{1/p}\|\Pi_J(\Delta_m^+f)\|_{H^p} \\
        &\lesssim_p
        2^{m/p}\|\Delta_m^+f\|_{H^p}
        =
        2^{-m\sigma}X_m.
\end{aligned}
\]
In particular, if
\(B_m(f)=\sum_{n\in I_m^+}a_n\), then
\[
        |B_m(f)|\lesssim_p2^{-m\sigma}X_m.
\]
Since \(\sigma>0\),
\((2^{-m\sigma})_{m\ge0}\in\ell^{p'}\), and Hölder's inequality gives
\[
        \sum_{m\ge0}|B_m(f)|
        \lesssim_{p,\alpha}
        \left(\sum_{m\ge0}X_m^p\right)^{1/p}
        \lesssim_{p,\alpha}
        \|f\|_{\mathcal D^p_\alpha}.
\]
For \(j=0\), the result follows from
\(|a_0|\lesssim_{p,\alpha}\|f\|_{\mathcal D^p_\alpha}\).
If \(j\ge1\), let \(M\ge0\) be the unique index such that
\(j\in I_M^+\). Then
\[
        \sum_{k=0}^j a_k
        =
        a_0+\sum_{m<M}B_m(f)
        +
        \sum_{\substack{k\in I_M^+\\ k\le j}}a_k.
\]
Hence
\[
\begin{aligned}
        \left|\sum_{k=0}^j a_k\right|
        &\le
        |a_0|+\sum_{m<M}|B_m(f)|
        +
        \left|
        \sum_{\substack{k\in I_M^+\\ k\le j}}a_k
        \right| \\
        &\lesssim_{p,\alpha}
        |a_0|+\sum_{m\ge0}|B_m(f)|
        +2^{-M\sigma}X_M \\
        &\lesssim_{p,\alpha}
        \|f\|_{\mathcal D^p_\alpha},
\end{aligned}
\]
uniformly in \(j\). This proves the second estimate.
\end{proof}

\begin{lemma}
\label{lem:prefix-block-estimate}
Let \(1<p,q<\infty\), \(-1<\alpha<p-2\), and let
\(f(z)=\sum_{k\ge0}a_kz^k\in\mathcal D^p_\alpha\). Define
\(\lambda_j(f)=\sum_{k=0}^{j+1}a_k\). Then, for every analytic
polynomial \(g\) and every \(m\ge0\),
\[
        \|\Delta_m^+T_{\lambda(f)}g\|_{H^q}
        \le
        C_{p,q,\alpha}
        \|f\|_{\mathcal D^p_\alpha}
        \|\Delta_m^+g\|_{H^q}.
\]
\end{lemma}

\begin{proof}
Put \(g_m=\Delta_m^+g\). For \(n\in I_m^+\), write
\(\lambda_n(f)=A_m(f)+P_{m,n}(f)\), where
\(A_m(f)=\sum_{k=0}^{2^m-1}a_k\) and
\(P_{m,n}(f)=\sum_{k=2^m}^{n+1}a_k\). Then
\(\Delta_m^+T_{\lambda(f)}g=A_m(f)g_m+W_m\), where
\[
        W_m(z)=\sum_{n\in I_m^+}
        P_{m,n}(f)\widehat g(n)z^n.
\]
The second estimate in Lemma~\ref{lem:prefix-estimates} gives
\(|A_m(f)|\lesssim_{p,\alpha}\|f\|_{\mathcal D^p_\alpha}\), and hence
\[
        \|A_m(f)g_m\|_{H^q}
        \lesssim_{p,\alpha}
        \|f\|_{\mathcal D^p_\alpha}\|g_m\|_{H^q}.
\]

To estimate \(W_m\), let
\(J_m=\{2^m,\ldots,2^{m+1}\}\),
\(u_m=\sum_{k\in J_m}a_kz^k\), and
\(\widetilde v_m=zg_m\). Then \(u_m\) and \(\widetilde v_m\) are
supported in \(J_m\),
\(\widehat{\widetilde v_m}(2^m)=0\), and
\(\widehat{\widetilde v_m}(\ell)=\widehat g(\ell-1)\) for
\(2^m+1\le\ell\le2^{m+1}\). By the definition of
\(\mathcal T_{J_m}\),
\[
\begin{aligned}
        \mathcal T_{J_m}(u_m,\widetilde v_m)(z)
        &=
        \sum_{\ell=2^m+1}^{2^{m+1}}
        \left(
        \sum_{k=2^m}^{\ell}a_k
        \right)
        \widehat g(\ell-1)z^\ell \\
        &=
        z\sum_{n\in I_m^+}
        \left(
        \sum_{k=2^m}^{n+1}a_k
        \right)
        \widehat g(n)z^n
        =
        zW_m(z).
\end{aligned}
\]
Since multiplication by \(z\) is isometric on \(H^q\),
Lemma~\ref{lem:local-triangular-estimate} and
\(|J_m|\asymp2^m\) give
\[
        \|W_m\|_{H^q}
        \lesssim_{p,q}
        2^{m/p}\|u_m\|_{H^p}\|g_m\|_{H^q}.
\]
Put \(\sigma=1-(\alpha+2)/p>0\), and let \(X_m\) be as in
Lemma~\ref{lem:prefix-estimates}. Since
\(u_m=\Delta_m^+f+
\Pi_{\{2^{m+1}\}}(\Delta_{m+1}^+f)\), Lemmas
\ref{lem:Hp-interval-projections} and
\ref{lem:prefix-estimates} give
\[
\begin{aligned}
        2^{m/p}\|u_m\|_{H^p}
        &\lesssim_p
        2^{m/p}
        \bigl(
        \|\Delta_m^+f\|_{H^p}
        +
        \|\Delta_{m+1}^+f\|_{H^p}
        \bigr) \\
        &\lesssim_{p,\alpha}
        2^{-m\sigma}(X_m+X_{m+1})
        \lesssim_{p,\alpha}
        \|f\|_{\mathcal D^p_\alpha}.
\end{aligned}
\]
Consequently,
\[
        \|W_m\|_{H^q}
        \lesssim_{p,q,\alpha}
        \|f\|_{\mathcal D^p_\alpha}\|g_m\|_{H^q}.
\]
Combining this with the estimate for \(A_m(f)g_m\) proves the lemma.
\end{proof}

\begin{proposition}
\label{prop:prefix-multiplier}
Let \(1<p,q<\infty\), \(\alpha,\beta>-1\), \(p>\alpha+2\), and
\(f(z)=\sum_{k\ge0}a_kz^k\in\mathcal D^p_\alpha\). Put
\(\lambda_j(f)=\sum_{k=0}^{j+1}a_k\). Then \(T_{\lambda(f)}\) is
bounded on \(A^q_\beta\), and
\[
        \|T_{\lambda(f)}g\|_{A^q_\beta}
        \le
        C_{p,q,\alpha,\beta}
        \|f\|_{\mathcal D^p_\alpha}
        \|g\|_{A^q_\beta},
        \qquad g\in A^q_\beta.
\]
\end{proposition}

\begin{proof}
It suffices to prove the estimate for analytic polynomials \(g\).
For \(h=T_{\lambda(f)}g\), we have
\(\widehat h(0)=\lambda_0(f)\widehat g(0)\). Hence
Proposition~\ref{prop:bergman-dyadic-model}, with the constant term
separated, and Lemmas~\ref{lem:prefix-estimates} and
\ref{lem:prefix-block-estimate} give
\[
\begin{aligned}
        \|T_{\lambda(f)}g\|_{A^q_\beta}^q
        &\lesssim_{q,\beta}
        |\lambda_0(f)\widehat g(0)|^q
        +
        \sum_{m\ge0}
        2^{-m(\beta+1)}
        \|\Delta_m^+T_{\lambda(f)}g\|_{H^q}^q \\
        &\lesssim_{p,q,\alpha,\beta}
        \|f\|_{\mathcal D^p_\alpha}^q
        \left(
        |\widehat g(0)|^q
        +
        \sum_{m\ge0}
        2^{-m(\beta+1)}
        \|\Delta_m^+g\|_{H^q}^q
        \right) \\
        &\lesssim_{p,q,\alpha,\beta}
        \|f\|_{\mathcal D^p_\alpha}^q
        \|g\|_{A^q_\beta}^q.
\end{aligned}
\]
Taking \(q\)-th roots and using density of analytic polynomials in
\(A^q_\beta\) proves the result.
\end{proof}

\begin{proof}[Proof of Theorem~\ref{thm:intro-below-critical}]
By Proposition~\ref{prop:dirichlet-dyadic-model},
\(\mathcal D^q_\beta=B^{1-(\beta+1)/q}_{q,q}\) with equivalent norms.
Thus it remains to prove the characterization in terms of
\(\mathcal D^q_\beta\), together with the operator norm estimate and
compactness.

If \(\mathcal R_{(\eta)}\) is bounded, then
\(\mathcal R_{(\eta)}1=F_\eta\). Since
\(\|1\|_{\mathcal D^p_\alpha}=1\),
\[
\|F_\eta\|_{\mathcal D^q_\beta}
\le
\|\mathcal R_{(\eta)}\|_{
\mathcal D^p_\alpha\to\mathcal D^q_\beta}.
\]

Conversely, assume that \(F_\eta\in\mathcal D^q_\beta\). For
\(f(z)=\sum_{k\ge0}a_kz^k\), put
\(\lambda_j(f)=\sum_{k=0}^{j+1}a_k\). If \(f\) is an analytic
polynomial, then
\[
(\mathcal R_{(\eta)}f)'
=
\sum_{j\ge0}(j+1)\eta_{j+1}\lambda_j(f)z^j
=
T_{\lambda(f)}(F_\eta').
\]
Since
\((\mathcal R_{(\eta)}f)(0)=\eta_0a_0\), with
\(|\eta_0|\le\|F_\eta\|_{\mathcal D^q_\beta}\) and
\(|a_0|\le\|f\|_{\mathcal D^p_\alpha}\),
Proposition~\ref{prop:prefix-multiplier} gives
\[
\begin{aligned}
\|\mathcal R_{(\eta)}f\|_{\mathcal D^q_\beta}
&\lesssim_{q,\beta}
|\eta_0a_0|
+
\|T_{\lambda(f)}(F_\eta')\|_{A^q_\beta} \\
&\lesssim_{p,q,\alpha,\beta}
\|F_\eta\|_{\mathcal D^q_\beta}
\|f\|_{\mathcal D^p_\alpha}.
\end{aligned}
\]
By density of analytic polynomials in \(\mathcal D^p_\alpha\),
\(\mathcal R_{(\eta)}\) extends uniquely to a bounded operator from
\(\mathcal D^p_\alpha\) into \(\mathcal D^q_\beta\). Together with
the necessity estimate above, this gives the norm equivalence.

It remains to prove compactness. Choose polynomials \(P_N\) such that
\(\|F_\eta'-P_N\|_{A^q_\beta}\to0\). For
\(P(z)=\sum_{j=0}^{M}c_jz^j\), define \(K_P\) by
\((K_Pf)(0)=\eta_0a_0\) and
\((K_Pf)'=T_{\lambda(f)}P\). Explicitly,
\[
K_Pf(z)
=
\eta_0a_0
+
\sum_{j=0}^{M}
\frac{c_j}{j+1}
\left(\sum_{k=0}^{j+1}a_k\right)z^{j+1}.
\]
The Taylor coefficient functionals are continuous on
\(\mathcal D^p_\alpha\), so \(K_P\) is bounded. Moreover,
\(\operatorname{Ran}K_P\subset
\operatorname{span}\{1,z,\ldots,z^{M+1}\}\), and hence \(K_P\) has
finite rank.
For every analytic polynomial \(f\),
\((\mathcal R_{(\eta)}f-K_Pf)(0)=0\) and
\[
(\mathcal R_{(\eta)}f-K_Pf)'
=
T_{\lambda(f)}(F_\eta'-P).
\]
Hence Proposition~\ref{prop:prefix-multiplier} gives
\[
\|\mathcal R_{(\eta)}f-K_Pf\|_{\mathcal D^q_\beta}
\lesssim_{p,q,\alpha,\beta}
\|f\|_{\mathcal D^p_\alpha}
\|F_\eta'-P\|_{A^q_\beta}.
\]
By density,
\[
\|\mathcal R_{(\eta)}-K_P\|_{
\mathcal D^p_\alpha\to\mathcal D^q_\beta}
\lesssim_{p,q,\alpha,\beta}
\|F_\eta'-P\|_{A^q_\beta}.
\]
Taking \(P=P_N\) gives
\[
\|\mathcal R_{(\eta)}-K_{P_N}\|_{
\mathcal D^p_\alpha\to\mathcal D^q_\beta}
\longrightarrow0.
\]
Since each \(K_{P_N}\) has finite rank,
\(\mathcal R_{(\eta)}\) is compact.
\end{proof}


\section{The case
\texorpdfstring{\(\alpha=p-2\)}{alpha = p - 2}}
\label{sec:critical-case}

Throughout this section, \(1<p,q<\infty\) and \(\beta>-1\). Put
\[
w_j
:=
2^{j(q-\beta-1)}
\|\Delta_jF_\eta\|_{H^q}^q,
\qquad
W_J:=\sum_{j\ge J}w_j.
\]

\begin{proof}[Proof of Theorem~\ref{thm:intro-critical}]
By Corollary~\ref{cor:critical-truncated-besov-norms}, it is enough
to prove the characterizations and norm estimates in terms of
\(G_{p,q,\beta}(\eta)\) and \(W_J(\eta)\).

We first prove boundedness and the operator norm estimates.
By Proposition~\ref{prop:dirichlet-dyadic-model},
\begin{equation}
\label{eq:critical-source-dyadic}
        \|f\|_{\mathcal D^p_{p-2}}^p
        \asymp_p
        |f(0)|^p+
        \sum_{j\ge0}2^j\|\Delta_jf\|_{H^p}^p,
\end{equation}
while
\begin{equation}
\label{eq:target-dirichlet-dyadic}
        \|h\|_{\mathcal D^q_\beta}^q
        \asymp_{q,\beta}
        |h(0)|^q+
        \sum_{j\ge0}
        2^{j(q-\beta-1)}\|\Delta_jh\|_{H^q}^q.
\end{equation}

For sufficiency, by density it is enough to consider analytic
polynomials. For \(n\in I_j\),
\[
        \sum_{k=0}^{n}\widehat f(k)
        =
        \sum_{m<j}\sum_{k\in I_m}\widehat f(k)
        +
        \sum_{\substack{k\in I_j\\ k\le n}}\widehat f(k).
\]
Therefore
\begin{equation}
\label{eq:rhaly-dyadic-block-decomposition}
\begin{aligned}
        \Delta_j(\mathcal R_{(\eta)}f)
        &=
        \left(
        \sum_{m<j}\sum_{k\in I_m}\widehat f(k)
        \right)\Delta_jF_\eta\\
        &\quad+
        \sum_{n\in I_j}\eta_n
        \left(
        \sum_{\substack{k\in I_j\\ k\le n}}\widehat f(k)
        \right)z^n.
\end{aligned}
\end{equation}
By Lemma~\ref{lem:block-summation-estimate},
\[
        \left|\sum_{k\in I_m}\widehat f(k)\right|
        \le
        C_p2^{m/p}\|\Delta_mf\|_{H^p},
\]
while Lemma~\ref{lem:local-triangular-estimate} gives
\[
        \left\|
        \sum_{n\in I_j}\eta_n
        \left(
        \sum_{\substack{k\in I_j\\ k\le n}}\widehat f(k)
        \right)z^n
        \right\|_{H^q}
        \le
        C_{p,q}2^{j/p}
        \|\Delta_jf\|_{H^p}\|\Delta_jF_\eta\|_{H^q}.
\]
Consequently,
\begin{equation}
\label{eq:rhaly-basic-block-estimate}
        \|\Delta_j(\mathcal R_{(\eta)}f)\|_{H^q}
        \lesssim_{p,q}
        \|\Delta_jF_\eta\|_{H^q}
        \sum_{m\le j}2^{m/p}\|\Delta_mf\|_{H^p}.
\end{equation}
Set \(x_m=2^{m/p}\|\Delta_mf\|_{H^p}\). Then
\[
        2^{j(q-\beta-1)}
        \|\Delta_j(\mathcal R_{(\eta)}f)\|_{H^q}^q
        \lesssim_{p,q}
        w_j\left(\sum_{m\le j}x_m\right)^q.
\]
If \(q<p\), Lemma~\ref{lem:integrated-tail-Hardy} and
\eqref{eq:critical-source-dyadic} give
\[
\begin{aligned}
        \sum_{j\ge0}
        2^{j(q-\beta-1)}
        \|\Delta_j(\mathcal R_{(\eta)}f)\|_{H^q}^q
        &\lesssim_{p,q}
        G_{p,q,\beta}(\eta)
        \left(\sum_{m\ge0}x_m^p\right)^{q/p}\\
        &\lesssim_{p,q}
        G_{p,q,\beta}(\eta)
        \|f\|_{\mathcal D^p_{p-2}}^q.
\end{aligned}
\]
If \(p\le q\), Lemma~\ref{lem:mixed-tail-Hardy} and
\eqref{eq:critical-source-dyadic} give
\[
\begin{aligned}
        \sum_{j\ge0}
        2^{j(q-\beta-1)}
        \|\Delta_j(\mathcal R_{(\eta)}f)\|_{H^q}^q
        &\lesssim_{p,q}
        \sup_{J\ge0}(J+1)^{q/p'}W_J(\eta)
        \left(\sum_{m\ge0}x_m^p\right)^{q/p}\\
        &\lesssim_{p,q}
        \sup_{J\ge0}(J+1)^{q/p'}W_J(\eta)
        \|f\|_{\mathcal D^p_{p-2}}^q.
\end{aligned}
\]
Moreover,
\[
        |\eta_0|^q
        \le
        \|\Delta_0F_\eta\|_{H^q}^q
        =
        w_0
        \le
        W_0
        \le
        \begin{cases}
        G_{p,q,\beta}(\eta), & q<p,\\[1mm]
        \displaystyle
        \sup_{J\ge0}(J+1)^{q/p'}W_J(\eta), & p\le q.
        \end{cases}
\]
Since \((\mathcal R_{(\eta)}f)(0)=\eta_0f(0)\), the constant term
satisfies the same upper bound. This proves sufficiency and the upper
operator norm estimates.

We turn to necessity. Assume that \(\mathcal R_{(\eta)}\) is bounded.
Taking \(f\equiv1\) gives
\(F_\eta=\mathcal R_{(\eta)}1\in\mathcal D^q_\beta\), and hence
\[
        W_0
        \lesssim_{q,\beta}
        \|F_\eta\|_{\mathcal D^q_\beta}^q
        \le
        \|\mathcal R_{(\eta)}\|^q.
\]

Suppose first that \(q<p\). For \(m\ge0\), put
\[
        I_m^+=\{2^m,\ldots,2^{m+1}-1\},
        \qquad
        u_m^+=|I_m^+|^{-1}\sum_{k\in I_m^+}z^k.
\]
Let \(x=(x_m)_{m\ge0}\) be a finitely supported nonnegative sequence,
and define
\[
        f_x
        =
        x_0+\sum_{m\ge0}x_{m+1}u_m^+.
\]
Then \(\Delta_0f_x=x_0+x_1z\) and
\(\Delta_mf_x=x_{m+1}u_m^+\) for \(m\ge1\). Since
\(f_x(0)=x_0\),
\(\|\Delta_0f_x\|_{H^p}^p\lesssim_p x_0^p+x_1^p\), and
Lemma~\ref{lem:normalized-interval-averages} gives
\(\|u_m^+\|_{H^p}^p\lesssim_p2^{-m}\), we obtain from
\eqref{eq:critical-source-dyadic}
\[
        \|f_x\|_{\mathcal D^p_{p-2}}^p
        \lesssim_p
        x_0^p+x_1^p+
        \sum_{m\ge1}
        2^m x_{m+1}^p\|u_m^+\|_{H^p}^p
        \lesssim_p
        \sum_{m\ge0}x_m^p.
\]
Thus
\[
        \|f_x\|_{\mathcal D^p_{p-2}}
        \lesssim_p
        \|x\|_{\ell^p}.
\]

For \(j\ge1\), set \(S_j=\sum_{m=0}^j x_m\). If \(n\in I_j\), then
\[
        \sum_{k=0}^{n}\widehat f_x(k)
        =
        S_j+x_{j+1}\theta_{j,n},
        \qquad
        \theta_{j,n}
        =
        \frac{n-2^j+1}{|I_j|}\in(0,1].
\]
Define \(\rho_j=(\rho_{j,n})_{n\ge0}\) by
\[
        \rho_{j,n}
        =
        \begin{cases}
        \displaystyle
        \frac{S_j}{S_j+x_{j+1}\theta_{j,n}},
            & n\in I_j \ \text{and}\ S_j>0, \\[6pt]
        0,  & \text{otherwise}.
        \end{cases}
\]
If \(S_j>0\), then \((\rho_{j,n})_{n\in I_j}\) is nonincreasing,
takes values in \((0,1]\), and vanishes outside \(I_j\). Hence its
total variation is at most \(2\). If \(S_j=0\), then
\(\rho_j\equiv0\). In either case,
\[
        T_{\rho_j}\Delta_j(\mathcal R_{(\eta)}f_x)
        =
        S_j\Delta_jF_\eta.
\]
Lemma~\ref{lem:bv-multiplier} therefore gives
\[
        S_j\|\Delta_jF_\eta\|_{H^q}
        \lesssim_q
        \|\Delta_j(\mathcal R_{(\eta)}f_x)\|_{H^q}.
\]
Taking \(q\)-th powers, multiplying by
\(2^{j(q-\beta-1)}\), and summing over \(j\ge1\), we obtain from
\eqref{eq:target-dirichlet-dyadic}
\[
\begin{aligned}
        \sum_{j\ge1}w_jS_j^q
        &\lesssim_{q,\beta}
        \|\mathcal R_{(\eta)}f_x\|_{\mathcal D^q_\beta}^q\\
        &\le
        \|\mathcal R_{(\eta)}\|^q
        \|f_x\|_{\mathcal D^p_{p-2}}^q
        \lesssim_{p,q}
        \|\mathcal R_{(\eta)}\|^q\|x\|_{\ell^p}^q.
\end{aligned}
\]
Since \(S_0=x_0\) and
\(w_0\le W_0\lesssim_{q,\beta}\|\mathcal R_{(\eta)}\|^q\),
the same estimate holds with the sum starting at \(j=0\).
For a general nonnegative sequence \(x\), apply the preceding
estimate to
\(x^{(N)}=(x_0,\ldots,x_N,0,\ldots)\) and let \(N\to\infty\).
By monotone convergence, the Hardy inequality in
Lemma~\ref{lem:integrated-tail-Hardy} holds with constant
\(C_{p,q,\beta}\|\mathcal R_{(\eta)}\|\). Hence
\[
        G_{p,q,\beta}(\eta)=B^q
        \lesssim_{p,q,\beta}
        \|\mathcal R_{(\eta)}\|^q.
\]
This proves necessity and the lower operator norm estimate when
\(q<p\).

Suppose now that \(p\le q\). For \(m\ge0\), let
\[
        u_m=|I_m|^{-1}\sum_{k\in I_m}z^k,
        \qquad
        \varphi_J=J^{-1/p}\sum_{m=0}^{J-1}u_m,
        \quad J\ge2.
\]
The dyadic supports of the \(u_m\) are disjoint, and
\(\Delta_m\varphi_J=J^{-1/p}u_m\) for \(0\le m<J\).
Moreover, \(u_0(0)=1/2\) and \(u_m(0)=0\) for \(m\ge1\), so
\(|\varphi_J(0)|^p=2^{-p}J^{-1}\). Hence
Lemma~\ref{lem:normalized-interval-averages} and
\eqref{eq:critical-source-dyadic} give
\[
        \|\varphi_J\|_{\mathcal D^p_{p-2}}^p
        \lesssim_p
        J^{-1}
        +
        J^{-1}\sum_{m=0}^{J-1}
        2^m\|u_m\|_{H^p}^p
        \lesssim_p1.
\]
Since
\(\bigcup_{m=0}^{J-1}I_m=\{0,\ldots,2^J-1\}\),
\(\varphi_J\) is a polynomial of degree \(2^J-1\). Hence, for
\(j\ge J\) and \(n\in I_j\),
\[
        \sum_{k=0}^n\widehat{\varphi_J}(k)
        =
        J^{-1/p}\sum_{m=0}^{J-1}1
        =
        J^{1/p'}.
\]
It follows that
\[
        \Delta_j(\mathcal R_{(\eta)}\varphi_J)
        =
        J^{1/p'}\Delta_jF_\eta,
        \qquad j\ge J.
\]
Using \eqref{eq:target-dirichlet-dyadic}, we obtain, for \(J\ge2\),
\[
\begin{aligned}
        J^{q/p'}W_J
        &=
        \sum_{j\ge J}2^{j(q-\beta-1)}
        \|\Delta_j(\mathcal R_{(\eta)}\varphi_J)\|_{H^q}^q\\
        &\lesssim_{q,\beta}
        \|\mathcal R_{(\eta)}\varphi_J\|_{\mathcal D^q_\beta}^q\\
        &\lesssim_{p,q}
        \|\mathcal R_{(\eta)}\|^q.
\end{aligned}
\]
Since \(J+1\asymp J\) for \(J\ge2\), while
\(W_0\lesssim_{q,\beta}\|\mathcal R_{(\eta)}\|^q\) and
\(W_1\le W_0\),
\[
        \sup_{J\ge0}(J+1)^{q/p'}W_J
        \lesssim_{p,q,\beta}
        \|\mathcal R_{(\eta)}\|^q.
\]
This completes the proof of the boundedness characterization and
operator norm estimates.

We next prove compactness. For \(N\ge0\), let
\[
        Q_Nf=\sum_{j=0}^N\Delta_jf,
\]
and define \(\eta^{(N)}\) by
\[
        \eta_n^{(N)}
        =
        \begin{cases}
        0, & n\in I_0\cup\cdots\cup I_N,\\
        \eta_n, & n\in I_j,\ j>N.
        \end{cases}
\]
Then
\[
        (I-Q_N)\mathcal R_{(\eta)}
        =
        \mathcal R_{(\eta^{(N)})}.
\]
The identity holds on analytic polynomials. Whenever
\(\mathcal R_{(\eta)}\) is bounded, it extends to
\(\mathcal D^p_{p-2}\) by density.
The dyadic weights corresponding
to \(\eta^{(N)}\) are
\[
        w_j^{(N)}
        =
        w_j\mathbf 1_{\{j>N\}},
        \qquad
        W_J^{(N)}
        =
        W_{\max\{J,N+1\}}.
\]
Suppose first that \(q<p\) and that
\(\mathcal R_{(\eta)}\) is bounded. Let \(r\) satisfy
\(1/r=1/q-1/p\), and put
\[
        B^r
        =
        \sum_{J\ge0}(J+1)^{r/q'}W_J^{r/q}.
\]
By the boundedness characterization proved above, \(B<\infty\).
For the truncated weights, put
\[
        B_N^r
        =
        \sum_{J\ge0}
        (J+1)^{r/q'}
        \bigl(W_J^{(N)}\bigr)^{r/q}.
\]
For each fixed \(J\), \(W_J^{(N)}\to0\) as \(N\to\infty\), while
\(0\le W_J^{(N)}\le W_J\). Hence dominated convergence gives
\(B_N\to0\) as \(N\to\infty\). The operator norm estimate already
proved gives
\[
        \|(I-Q_N)\mathcal R_{(\eta)}\|_{
        \mathcal D^p_{p-2}\to\mathcal D^q_\beta}^q
        \lesssim_{p,q,\beta}
        B_N^q
        \longrightarrow0
        \qquad (N\to\infty).
\]
Since \(Q_N\mathcal R_{(\eta)}\) has finite-dimensional range,
\(\mathcal R_{(\eta)}\) is compact. The reverse implication is immediate.

Now suppose that \(p\le q\). Put
\[
        B_J=(J+1)^{q/p'}W_J,
        \qquad\text{and}\qquad
        A_N=\sup_{L\ge0}(L+1)^{q/p'}
        W_{\max\{L,N+1\}}.
\]
Splitting the supremum according as \(L\le N+1\) or \(L>N+1\)
gives
\[
        A_N=\sup_{J\ge N+1}B_J.
\]
If \(B_J\to0\) as \(J\to\infty\), then
\(\sup_{J\ge0}B_J<\infty\), so the boundedness characterization above
gives
\[
\mathcal R_{(\eta)}
\in\mathcal B(\mathcal D^p_{p-2},\mathcal D^q_\beta).
\]
Moreover, \(A_N\to0\) as \(N\to\infty\). Applying the operator norm
estimate to \(\eta^{(N)}\), we obtain
\[
        \|(I-Q_N)\mathcal R_{(\eta)}\|_{
        \mathcal D^p_{p-2}\to\mathcal D^q_\beta}^q
        \lesssim_{p,q,\beta}
        A_N
        \longrightarrow0
        \qquad (N\to\infty).
\]
Thus \(\mathcal R_{(\eta)}\) is compact.

Conversely, suppose that \(\mathcal R_{(\eta)}\) is compact. For \(J\ge2\), let \(\varphi_J\) be as above. Then
\[
        \|\varphi_J\|_{\mathcal D^p_{p-2}}
        \lesssim_p 1,
        \qquad\text{and}\qquad
        \Delta_j(\mathcal R_{(\eta)}\varphi_J)
        =
        J^{1/p'}\Delta_jF_\eta,
        \quad j\ge J.
\]
Moreover, for every \(0<\rho<1\),
\[
        \sup_{|z|\le\rho}|\varphi_J(z)|
        \le C_\rho J^{-1/p},
\]
so \(\varphi_J\to0\) uniformly on compact subsets of \(\D\) as
\(J\to\infty\). Since \((\varphi_J)\) is bounded in the reflexive space
\(\mathcal D^p_{p-2}\), every subsequence has a weakly convergent
subsequence, and the local uniform convergence together with the
continuity of point evaluations forces every weak limit to be zero.
Hence
\[
        \varphi_J\rightharpoonup0
        \qquad\text{in }\mathcal D^p_{p-2}
        \quad (J\to\infty).
\]
If \(\mathcal R_{(\eta)}\) is compact, then
\(\|\mathcal R_{(\eta)}\varphi_J\|_{\mathcal D^q_\beta}\to0\) as
\(J\to\infty\). Using \eqref{eq:target-dirichlet-dyadic} and
\(J+1\asymp J\), we obtain
\[
        B_J
        \lesssim_{p,q,\beta}
        \|\mathcal R_{(\eta)}\varphi_J\|_{\mathcal D^q_\beta}^q
        \longrightarrow0
        \qquad (J\to\infty).
\]
This proves the compactness characterization.

It remains to estimate the essential norm when \(p\le q\). Since
\(Q_N\mathcal R_{(\eta)}\) has finite rank and
\(A_N=\sup_{J\ge N+1}B_J\),
\[
        \|\mathcal R_{(\eta)}\|_e^q
        \le
        \|(I-Q_N)\mathcal R_{(\eta)}\|^q
        \lesssim_{p,q,\beta}
        A_N.
\]
Letting \(N\to\infty\), we obtain
\[
        \|\mathcal R_{(\eta)}\|_e^q
        \lesssim_{p,q,\beta}
        \limsup_{J\to\infty}B_J.
\]

For the reverse estimate, let
\(K:\mathcal D^p_{p-2}\to\mathcal D^q_\beta\) be compact, and choose
\(J_k\to\infty\) such that
\(B_{J_k}\to\limsup_{J\to\infty}B_J\).
Since
\[
        B_{J_k}^{1/q}
        \lesssim_{p,q,\beta}
        \|\mathcal R_{(\eta)}\varphi_{J_k}\|_{\mathcal D^q_\beta},
\]
while
\(\|K\varphi_{J_k}\|_{\mathcal D^q_\beta}\to0\) and
\(\sup_J\|\varphi_J\|_{\mathcal D^p_{p-2}}\lesssim_p1\), we obtain
\[
\begin{aligned}
        \|\mathcal R_{(\eta)}-K\|
        &\gtrsim_p
        \limsup_{k\to\infty}
        \|(\mathcal R_{(\eta)}-K)\varphi_{J_k}\|_{\mathcal D^q_\beta}\\
        &\gtrsim_{p,q,\beta}
        \left(
        \limsup_{J\to\infty}B_J
        \right)^{1/q}.
\end{aligned}
\]
Taking the infimum over compact \(K\) gives the reverse estimate and
completes the proof.
\end{proof}

\begin{remark}
Taking \(q=p\) and \(\beta=p-2\) in
Theorem~\ref{thm:intro-critical}\textup{(ii)} gives the boundedness
and compactness characterizations, together with the essential norm
estimate, for \(\mathcal R_{(\eta)}\) on
\(B^p=\mathcal D^p_{p-2}\). When \(p=2\), orthogonality and
\(n\asymp2^j\) on \(I_j\) give
\(2^j\|\Delta_jF_\eta\|_{H^2}^2
\asymp\sum_{n\in I_j}n|\eta_n|^2\), \(j\ge1\).
Hence, since \(J+1\asymp\log(N+2)\) whenever
\(2^J\le N<2^{J+1}\),
\[
        \sup_{J\ge0}(J+1)W_J(\eta)<\infty
        \quad\Longleftrightarrow\quad
        \sup_{N\ge1}
        \log(N+2)\sum_{n\ge N}n|\eta_n|^2<\infty.
\]
This recovers the known boundedness characterization on the classical
Dirichlet space; see
\cite[Theorem~1.3]{BaoGuoSunWang2024} and
\cite{BlascoGalanopoulosGirela2026}.
\end{remark}

\begin{corollary}
\label{cor:critical-phase-transition}
Let \(1<p,q<\infty\) and \(\beta>-1\). For
\(\mathcal R_{(\eta)}:\mathcal D^p_{p-2}\to\mathcal D^q_\beta\),
the following hold:
\begin{enumerate}[label=\textup{(\roman*)},leftmargin=2.2em]
\item if \(q<p\), every bounded \(\mathcal R_{(\eta)}\) is compact;
\item if \(p\le q\), there exists a bounded noncompact
\(\mathcal R_{(\eta)}\).
\end{enumerate}
\end{corollary}

\begin{proof}
Part~\textup{(i)} follows from
Theorem~\ref{thm:intro-critical}\textup{(i)}.

Suppose that \(p\le q\), and put \(a=q/p'\). For \(j\ge1\), set
\[
        U_j=(j+1)^{-a},
        \qquad
        v_j=U_j-U_{j+1}>0.
\]
Let \(n_j=2^j\in I_j\), and define
\[
        \eta_{n_j}
        =
        2^{-j(q-\beta-1)/q}v_j^{1/q},
        \qquad
        \eta_n=0
        \quad (n\ne n_j).
\]
Since \(|\eta_{n_j}|^{1/n_j}\to1\), the lacunary series
\[
        F_\eta(z)=\sum_{j\ge1}\eta_{n_j}z^{n_j}
\]
belongs to \(\Hol(\D)\). Moreover,
\[
        2^{j(q-\beta-1)}
        \|\Delta_jF_\eta\|_{H^q}^q
        =
        v_j,
        \qquad j\ge1.
\]
Since
\(\sum_{j\ge J}v_j=U_J\),
\[
        (J+1)^{q/p'}
        \sum_{j\ge J}
        2^{j(q-\beta-1)}
        \|\Delta_jF_\eta\|_{H^q}^q
        =
        1,
        \qquad J\ge1.
\]
Thus \(\mathcal R_{(\eta)}\) is bounded but not compact by
Theorem~\ref{thm:intro-critical}\textup{(ii)}.
\end{proof}

\section{The case
\texorpdfstring{\(\alpha>p-2\)}{alpha > p - 2}}
\label{sec:above-critical}

Throughout this section, \(1<p,q<\infty\), \(\alpha>p-2\), and
\(\beta>-1\). Put
\[
        \delta=\frac{\alpha+2-p}{p}>0,
        \qquad
        s=\frac{\alpha+2}{p}-\frac{\beta+1}{q},
        \qquad
        A_j=2^{js}\|\Delta_jF_\eta\|_{H^q}.
\]

\begin{lemma}
\label{lem:above-critical-prefix}
Let \(f\) be an analytic polynomial, and put
\[
        X_m
        =
        2^{m(1-(\alpha+1)/p)}\|\Delta_mf\|_{H^p},
        \qquad
        Y_j
        =
        \sum_{m\le j}2^{-\delta(j-m)}X_m.
\]
Then
\[
        \|Y\|_{\ell^p}
        \lesssim_{p,\alpha}
        \|f\|_{\mathcal D^p_\alpha},
        \qquad
        2^{j(1-(\beta+1)/q)}
        \|\Delta_j(\mathcal R_{(\eta)}f)\|_{H^q}
        \lesssim_{p,q}
        A_jY_j
\]
for every \(j\ge0\).
\end{lemma}

\begin{proof}
Since \((2^{-k\delta})_{k\ge0}\in\ell^1\), the discrete Young
inequality and Proposition~\ref{prop:dirichlet-dyadic-model} give
\[
        \|Y\|_{\ell^p}
        \lesssim_{\delta}
        \|X\|_{\ell^p}
        \lesssim_{p,\alpha}
        \|f\|_{\mathcal D^p_\alpha}.
\]
The estimate \eqref{eq:rhaly-basic-block-estimate} was obtained
without using the assumption \(\alpha=p-2\), and therefore remains
valid here. Hence, since
\[
        \sum_{m\le j}2^{m/p}\|\Delta_mf\|_{H^p}
        =
        2^{j\delta}Y_j,
\]
we obtain
\[
        2^{j(1-(\beta+1)/q)}
        \|\Delta_j(\mathcal R_{(\eta)}f)\|_{H^q}
        \lesssim_{p,q}
        2^{j(1-(\beta+1)/q+\delta)}
        \|\Delta_jF_\eta\|_{H^q}Y_j.
\]
Since \(1-(\beta+1)/q+\delta=s\), the last expression is
\(A_jY_j\). This proves the second estimate.
\end{proof}

We use the following standard characterization of diagonal
multipliers between sequence spaces.

\begin{lemma}
\label{lem:diagonal-sequence-multiplier}
Let \(1<p,q<\infty\), and let
\(a=(a_j)_{j\ge0}\) be nonnegative. The diagonal map
\(M_ax=(a_jx_j)_j\) is bounded from \(\ell^p\) into \(\ell^q\) if and
only if \(a\in\ell^r\), where \(1/r=1/q-1/p\), when \(q<p\), and if
and only if \(a\in\ell^\infty\) when \(p\le q\). Moreover,
\[
        \|M_a\|_{\ell^p\to\ell^q}
        =
        \begin{cases}
        \|a\|_{\ell^r},& q<p,\\
        \|a\|_{\ell^\infty},& p\le q.
        \end{cases}
\]
\end{lemma}

\begin{proof}
Suppose first that \(q<p\). If \(a\in\ell^r\), where
\(1/r=1/q-1/p\), H\"older's inequality gives
\[
        \|M_ax\|_{\ell^q}
        \le
        \|a\|_{\ell^r}\|x\|_{\ell^p}.
\]
Conversely, assume that \(M_a:\ell^p\to\ell^q\) is bounded. For a
finite set \(E\), set
\(x_j=a_j^{r/p}\mathbf1_E(j)\). Since \(r/q=1+r/p\),
\[
        \left(\sum_{j\in E}a_j^r\right)^{1/q}
        =
        \|M_ax\|_{\ell^q}
        \le
        \|M_a\|
        \left(\sum_{j\in E}a_j^r\right)^{1/p}.
\]
Hence
\[
        \left(\sum_{j\in E}a_j^r\right)^{1/r}
        \le
        \|M_a\|.
\]
Taking the supremum over all finite \(E\) gives
\(a\in\ell^r\) and
\(\|a\|_{\ell^r}\le\|M_a\|\).

If \(p\le q\), testing \(M_a\) on the unit vectors gives
\(\|a\|_{\ell^\infty}\le\|M_a\|\). Conversely,
\[
        \|M_ax\|_{\ell^q}
        \le
        \|a\|_{\ell^\infty}\|x\|_{\ell^q}
        \le
        \|a\|_{\ell^\infty}\|x\|_{\ell^p},
\]
since \(\ell^p\subset\ell^q\). This proves the result.
\end{proof}

\begin{theorem}
\label{thm:above-critical-Rhaly}
Let \(F_\eta\in\Hol(\D)\), and write
\(A_j=2^{js}\|\Delta_jF_\eta\|_{H^q}\), \(j\ge0\).
\begin{enumerate}[label=\textup{(\roman*)},leftmargin=2.2em]

\item If \(q<p\) and \(1/r=1/q-1/p\), then
\[
\begin{aligned}
\mathcal R_{(\eta)}
\in\mathcal B(\mathcal D^p_\alpha,\mathcal D^q_\beta)
&\Longleftrightarrow
(A_j)_{j\ge0}\in\ell^r, \\
\|\mathcal R_{(\eta)}\|_{
\mathcal D^p_\alpha\to\mathcal D^q_\beta}
&\asymp_{p,q,\alpha,\beta}
\|(A_j)\|_{\ell^r}.
\end{aligned}
\]

\item If \(p\le q\), then
\[
\begin{aligned}
\mathcal R_{(\eta)}
\in\mathcal B(\mathcal D^p_\alpha,\mathcal D^q_\beta)
&\Longleftrightarrow
\sup_{j\ge0}A_j<\infty, \\
\|\mathcal R_{(\eta)}\|_{
\mathcal D^p_\alpha\to\mathcal D^q_\beta}
&\asymp_{p,q,\alpha,\beta}
\sup_{j\ge0}A_j.
\end{aligned}
\]

\end{enumerate}
\end{theorem}

\begin{proof}
For sufficiency, let \(f\) be an analytic polynomial and let \(X\)
and \(Y\) be as in Lemma~\ref{lem:above-critical-prefix}. By
Proposition~\ref{prop:dirichlet-dyadic-model} and
Lemma~\ref{lem:above-critical-prefix},
\[
        \|\mathcal R_{(\eta)}f\|_{\mathcal D^q_\beta}
        \lesssim_{p,q,\alpha,\beta}
        |\eta_0\widehat f(0)|
        +
        \|(A_jY_j)_j\|_{\ell^q}.
\]
Since \(|\eta_0|\le A_0\),
Lemma~\ref{lem:diagonal-sequence-multiplier} and
\(\|Y\|_{\ell^p}\lesssim_{p,\alpha}
\|f\|_{\mathcal D^p_\alpha}\) give
\begin{equation}
\label{eq:above-critical-sufficiency}
        \|\mathcal R_{(\eta)}f\|_{\mathcal D^q_\beta}
        \lesssim_{p,q,\alpha,\beta}
        \begin{cases}
        \|A\|_{\ell^r}\|f\|_{\mathcal D^p_\alpha},
            & q<p,\\
        \|A\|_{\ell^\infty}\|f\|_{\mathcal D^p_\alpha},
            & p\le q.
        \end{cases}
\end{equation}
By density, \(\mathcal R_{(\eta)}\) extends uniquely to a bounded
operator from \(\mathcal D^p_\alpha\) into \(\mathcal D^q_\beta\),
and the upper operator norm estimate follows.

For necessity, put
\[
        u_m=|I_m|^{-1}\sum_{k\in I_m}z^k,
        \qquad
        h_m=2^{m\delta}u_m,
        \quad m\ge1.
\]
By Lemma~\ref{lem:normalized-interval-averages},
\(\|u_m\|_{H^p}\lesssim_p2^{-m/p}\). Since \(h_m\) is supported on
\(I_m\), Proposition~\ref{prop:dirichlet-dyadic-model} gives
\[
        \|h_m\|_{\mathcal D^p_\alpha}
        \lesssim_{p,\alpha}
        2^{m(1-(\alpha+1)/p+\delta-1/p)}
        \lesssim_{p,\alpha}1,
\]
because
\(1-(\alpha+1)/p+\delta-1/p=0\).

Suppose first that \(q<p\). For
\(\varepsilon\in\{0,1\}\), let
\[
        E_\varepsilon
        =
        \{j\ge2:j\equiv\varepsilon\pmod2\}.
\]
Given a finitely supported nonnegative sequence
\((x_j)_{j\in E_\varepsilon}\), set
\[
        f_x
        =
        \sum_{j\in E_\varepsilon}x_jh_{j-1}.
\]
Since
\(\Delta_{j-1}f_x=x_jh_{j-1}\) for
\(j\in E_\varepsilon\),
Proposition~\ref{prop:dirichlet-dyadic-model} gives
\[
        \|f_x\|_{\mathcal D^p_\alpha}^p
        \asymp_{p,\alpha}
        \sum_{j\in E_\varepsilon}
        2^{(j-1)(p-\alpha-1)}
        x_j^p\|h_{j-1}\|_{H^p}^p.
\]
Since each \(h_{j-1}\) is supported on \(I_{j-1}\),
\[
        2^{(j-1)(p-\alpha-1)}
        \|h_{j-1}\|_{H^p}^p
        \asymp_{p,\alpha}
        \|h_{j-1}\|_{\mathcal D^p_\alpha}^p,
\]
and hence
\[
        \|f_x\|_{\mathcal D^p_\alpha}
        \lesssim_{p,\alpha}
        \|(x_j)\|_{\ell^p(E_\varepsilon)}.
\]

Fix \(j\in E_\varepsilon\). Since the indices in
\(E_\varepsilon\) have the same parity, \(I_{\ell-1}\) lies entirely
before \(I_j\) when \(\ell\le j\) and entirely after \(I_j\) when
\(\ell>j\). Hence, for every \(n\in I_j\),
\[
        \sum_{k=0}^n\widehat f_x(k)
        =
        \sum_{\substack{\ell\in E_\varepsilon\\ \ell\le j}}
        x_\ell2^{(\ell-1)\delta}.
\]
The right hand side is independent of \(n\) and, since
\(x_\ell\ge0\), is at least \(x_j2^{(j-1)\delta}\). Therefore
\[
        \Delta_j(\mathcal R_{(\eta)}f_x)
        =
        \left(
        \sum_{\substack{\ell\in E_\varepsilon\\ \ell\le j}}
        x_\ell2^{(\ell-1)\delta}
        \right)\Delta_jF_\eta,
\]
and hence
\[
        2^{j(1-(\beta+1)/q)}
        \|\Delta_j(\mathcal R_{(\eta)}f_x)\|_{H^q}
        \ge
        2^{-\delta}A_jx_j.
\]
It follows that
\[
\begin{aligned}
        2^{-\delta}
        \|(A_jx_j)_{j\in E_\varepsilon}\|_{\ell^q(E_\varepsilon)}
        &\le
        \left(
        \sum_{j\in E_\varepsilon}
        2^{j(q-\beta-1)}
        \|\Delta_j(\mathcal R_{(\eta)}f_x)\|_{H^q}^q
        \right)^{1/q}\\
        &\lesssim_{q,\beta}
        \|\mathcal R_{(\eta)}f_x\|_{\mathcal D^q_\beta}.
\end{aligned}
\]
Thus
\[
        \|(A_jx_j)_{j\in E_\varepsilon}\|_{\ell^q(E_\varepsilon)}
        \lesssim_{p,q,\alpha,\beta}
        \|\mathcal R_{(\eta)}\|
        \|(x_j)_{j\in E_\varepsilon}\|_{\ell^p(E_\varepsilon)}.
\]
Since \(A_j\ge0\), replacing \(x_j\) by \(|x_j|\) shows that the same
estimate holds for every finitely supported scalar sequence on
\(E_\varepsilon\).
Lemma~\ref{lem:diagonal-sequence-multiplier} gives
\[
        \|(A_j)_{j\in E_\varepsilon}\|_{\ell^r(E_\varepsilon)}
        \lesssim_{p,q,\alpha,\beta}
        \|\mathcal R_{(\eta)}\|.
\]
Applying this estimate to both parity classes,
\[
        \|(A_j)_{j\ge2}\|_{\ell^r}
        \lesssim_{p,q,\alpha,\beta}
        \|\mathcal R_{(\eta)}\|.
\]
Moreover, since \(F_\eta=\mathcal R_{(\eta)}1\),
Proposition~\ref{prop:dirichlet-dyadic-model} gives
\[
        A_0+A_1
        \lesssim_{p,q,\alpha,\beta}
        \|F_\eta\|_{\mathcal D^q_\beta}
        \lesssim_{p,q,\alpha,\beta}
        \|\mathcal R_{(\eta)}\|.
\]
Hence
\begin{equation}
\label{eq:above-critical-lower-ell-r}
        \|A\|_{\ell^r}
        \lesssim_{p,q,\alpha,\beta}
        \|\mathcal R_{(\eta)}\|.
\end{equation}

Now suppose that \(p\le q\). For \(m\ge1\),
\[
        \Delta_{m+1}(\mathcal R_{(\eta)}h_m)
        =
        \left(
        \sum_{k\in I_m}\widehat h_m(k)
        \right)
        \Delta_{m+1}F_\eta
        =
        2^{m\delta}\Delta_{m+1}F_\eta.
\]
Since \(s=1-(\beta+1)/q+\delta\),
Proposition~\ref{prop:dirichlet-dyadic-model} and the uniform
boundedness of \((h_m)\) in \(\mathcal D^p_\alpha\) give
\begin{equation}
\label{eq:above-critical-block-lower}
\begin{aligned}
A_{m+1}
&=
2^\delta
2^{(m+1)(1-(\beta+1)/q)}
\|\Delta_{m+1}(\mathcal R_{(\eta)}h_m)\|_{H^q}\\
&\lesssim_{p,q,\alpha,\beta}
\|\mathcal R_{(\eta)}h_m\|_{\mathcal D^q_\beta}
\lesssim_{p,q,\alpha,\beta}
\|\mathcal R_{(\eta)}\|.
\end{aligned}
\end{equation}
Finally, since \(F_\eta=\mathcal R_{(\eta)}1\),
\[
        A_0+A_1
        \lesssim_{p,q,\alpha,\beta}
        \|F_\eta\|_{\mathcal D^q_\beta}
        \lesssim_{p,q,\alpha,\beta}
        \|\mathcal R_{(\eta)}\|.
\]
Consequently,
\begin{equation}
\label{eq:above-critical-lower-ell-infty}
        \|A\|_{\ell^\infty}
        \lesssim_{p,q,\alpha,\beta}
        \|\mathcal R_{(\eta)}\|.
\end{equation}
Combining \eqref{eq:above-critical-lower-ell-r} and
\eqref{eq:above-critical-lower-ell-infty} gives the lower operator
norm estimates.
\end{proof}

\begin{theorem}
\label{thm:above-critical-compact}
Under the hypotheses and notation of
Theorem~\ref{thm:above-critical-Rhaly}, the following hold.
\begin{enumerate}[label=\textup{(\roman*)},leftmargin=2.2em]

\item If \(q<p\), then
\[
        \mathcal R_{(\eta)}
        \in\mathcal B(\mathcal D^p_\alpha,\mathcal D^q_\beta)
        \Longleftrightarrow
        \mathcal R_{(\eta)}
        \in\mathcal K(\mathcal D^p_\alpha,\mathcal D^q_\beta).
\]

\item If \(p\le q\), then
\[
        \mathcal R_{(\eta)}
        \in\mathcal K(\mathcal D^p_\alpha,\mathcal D^q_\beta)
        \Longleftrightarrow
        \lim_{j\to\infty}A_j=0.
\]
Whenever \(\mathcal R_{(\eta)}\) is bounded,
\[
        \|\mathcal R_{(\eta)}\|_e
        \asymp_{p,q,\alpha,\beta}
        \limsup_{j\to\infty}A_j.
\]

\end{enumerate}
\end{theorem}

\begin{proof}
Suppose first that \(q<p\), let \(1/r=1/q-1/p\), and assume that
\(\mathcal R_{(\eta)}\) is bounded. By
Theorem~\ref{thm:above-critical-Rhaly},
\(A=(A_j)\in\ell^r\). For \(N\ge0\), let
\(Q_Nf=\sum_{j=0}^N\Delta_jf\). For analytic polynomials \(f\),
Proposition~\ref{prop:dirichlet-dyadic-model} and
Lemma~\ref{lem:above-critical-prefix} give
\[
\begin{aligned}
        \|(I-Q_N)\mathcal R_{(\eta)}f\|_{\mathcal D^q_\beta}
        &\lesssim_{q,\beta}
        \left(
        \sum_{j>N}
        2^{j(q-\beta-1)}
        \|\Delta_j(\mathcal R_{(\eta)}f)\|_{H^q}^q
        \right)^{1/q}\\
        &\lesssim_{p,q}
        \|(A_jY_j)_{j>N}\|_{\ell^q}\\
        &\le
        \|(A_j)_{j>N}\|_{\ell^r}
        \|(Y_j)_{j>N}\|_{\ell^p}\\
        &\lesssim_{p,\alpha}
        \|(A_j)_{j>N}\|_{\ell^r}
        \|f\|_{\mathcal D^p_\alpha}.
\end{aligned}
\]
By density,
\begin{equation}
\label{eq:above-critical-tail-r}
        \|(I-Q_N)\mathcal R_{(\eta)}\|_{
        \mathcal D^p_\alpha\to\mathcal D^q_\beta}
        \lesssim_{p,q,\alpha,\beta}
        \|(A_j)_{j>N}\|_{\ell^r}
        \longrightarrow0
        \qquad (N\to\infty).
\end{equation}
Since \(Q_N\mathcal R_{(\eta)}\) has finite rank for each \(N\),
\(\mathcal R_{(\eta)}\) is compact. The reverse implication is
immediate.

Now suppose that \(p\le q\). If \(\lim_{j\to\infty}A_j=0\), then
Theorem~\ref{thm:above-critical-Rhaly} first gives the boundedness of
\(\mathcal R_{(\eta)}\). As above, for analytic polynomials \(f\),
\[
\begin{aligned}
        \|(I-Q_N)\mathcal R_{(\eta)}f\|_{\mathcal D^q_\beta}
        &\lesssim_{p,q,\alpha,\beta}
        \|(A_jY_j)_{j>N}\|_{\ell^q}\\
        &\le
        \sup_{j>N}A_j
        \|(Y_j)_{j>N}\|_{\ell^q}\\
        &\le
        \sup_{j>N}A_j
        \|Y\|_{\ell^p}\\
        &\lesssim_{p,\alpha}
        \sup_{j>N}A_j
        \|f\|_{\mathcal D^p_\alpha}.
\end{aligned}
\]
Hence, by density,
\begin{equation}
\label{eq:above-critical-tail-infty}
        \|(I-Q_N)\mathcal R_{(\eta)}\|_{
        \mathcal D^p_\alpha\to\mathcal D^q_\beta}
        \lesssim_{p,q,\alpha,\beta}
        \sup_{j>N}A_j.
\end{equation}
The right hand side tends to zero, and
\(Q_N\mathcal R_{(\eta)}\) has finite rank. Thus
\(\mathcal R_{(\eta)}\) is compact.

For the converse, let \(h_m\) be the test functions used in the proof
of Theorem~\ref{thm:above-critical-Rhaly}. They are uniformly bounded
in \(\mathcal D^p_\alpha\), and, for every \(0<\rho<1\),
\[
        \sup_{|z|\le\rho}|h_m(z)|
        \lesssim_{\rho,p,\alpha}
        2^{m\delta}\rho^{2^m}
        \longrightarrow0
        \qquad (m\to\infty).
\]
Thus \(h_m\to0\) uniformly on compact subsets of \(\D\). Since
\((h_m)\) is bounded in the reflexive space \(\mathcal D^p_\alpha\),
every subsequence has a weakly convergent subsequence. The local
uniform convergence and continuity of point evaluations force every
weak limit to be zero. Hence
\[
        h_m\rightharpoonup0
        \qquad\text{in }\mathcal D^p_\alpha
        \quad (m\to\infty).
\]
If \(\mathcal R_{(\eta)}\) is compact, then
\(\|\mathcal R_{(\eta)}h_m\|_{\mathcal D^q_\beta}\to0\) as
\(m\to\infty\). By \eqref{eq:above-critical-block-lower},
\[
        A_{m+1}
        \lesssim_{p,q,\alpha,\beta}
        \|\mathcal R_{(\eta)}h_m\|_{\mathcal D^q_\beta}
        \longrightarrow0
        \qquad (m\to\infty).
\]
Hence \(A_j\to0\), proving the compactness characterization.

It remains to estimate the essential norm. By
\eqref{eq:above-critical-tail-infty},
\[
        \|\mathcal R_{(\eta)}\|_e
        \le
        \|\mathcal R_{(\eta)}-Q_N\mathcal R_{(\eta)}\|
        \lesssim_{p,q,\alpha,\beta}
        \sup_{j>N}A_j.
\]
Letting \(N\to\infty\), we obtain
\begin{equation}
\label{eq:above-critical-essential-upper}
        \|\mathcal R_{(\eta)}\|_e
        \lesssim_{p,q,\alpha,\beta}
        \limsup_{j\to\infty}A_j.
\end{equation}

For the reverse estimate, let
\(K:\mathcal D^p_\alpha\to\mathcal D^q_\beta\) be compact. Since
\(h_m\rightharpoonup0\),
\(\|Kh_m\|_{\mathcal D^q_\beta}\to0\) as \(m\to\infty\). Using
\eqref{eq:above-critical-block-lower} and the uniform boundedness of
\((h_m)\) in \(\mathcal D^p_\alpha\), we obtain
\[
\begin{aligned}
        A_{m+1}
        &\lesssim_{p,q,\alpha,\beta}
        \|\mathcal R_{(\eta)}h_m\|_{\mathcal D^q_\beta}\\
        &\lesssim_{p,q,\alpha,\beta}
        \|\mathcal R_{(\eta)}-K\|
        +
        \|Kh_m\|_{\mathcal D^q_\beta}.
\end{aligned}
\]
Taking the limit superior as \(m\to\infty\) gives
\[
        \limsup_{j\to\infty}A_j
        \lesssim_{p,q,\alpha,\beta}
        \|\mathcal R_{(\eta)}-K\|.
\]
Taking the infimum over all compact \(K\), we obtain
\begin{equation}
\label{eq:above-critical-essential-lower}
        \limsup_{j\to\infty}A_j
        \lesssim_{p,q,\alpha,\beta}
        \|\mathcal R_{(\eta)}\|_e.
\end{equation}
Combining \eqref{eq:above-critical-essential-upper} and
\eqref{eq:above-critical-essential-lower} gives the essential norm
estimate.
\end{proof}

\begin{proof}[Proof of Theorem~\ref{thm:intro-above-critical}]
The result follows from
Theorems~\ref{thm:above-critical-Rhaly}
and~\ref{thm:above-critical-compact}.
\end{proof}

\begin{corollary}
\label{cor:same-space-classification}
Let \(1<p<\infty\) and \(\alpha>-1\). Then the boundedness and
compactness of \(\mathcal R_{(\eta)}\) on \(\mathcal D^p_\alpha\)
are characterized as follows.
\begin{enumerate}[label=\textup{(\roman*)},leftmargin=2.2em]
\item If \(\alpha>p-2\), then
\[
        \mathcal R_{(\eta)}\in\mathcal B(\mathcal D^p_\alpha)
        \Longleftrightarrow
        F_\eta\in\Lambda^p_{1/p},
\]
and
\[
        \mathcal R_{(\eta)}\in\mathcal K(\mathcal D^p_\alpha)
        \Longleftrightarrow
        F_\eta\in\lambda^p_{1/p}.
\]

\item If \(\alpha=p-2\), then
\[
        \mathcal R_{(\eta)}
        \in\mathcal B(\mathcal D^p_{p-2})
        \Longleftrightarrow
        \sup_{J\ge0}(J+1)^{p-1}
        \sum_{j\ge J}
        2^j\|\Delta_jF_\eta\|_{H^p}^p<\infty,
\]
and
\[
        \mathcal R_{(\eta)}
        \in\mathcal K(\mathcal D^p_{p-2})
        \Longleftrightarrow
        \lim_{J\to\infty}
        (J+1)^{p-1}
        \sum_{j\ge J}
        2^j\|\Delta_jF_\eta\|_{H^p}^p
        =
        0.
\]

\item If \(-1<\alpha<p-2\), then
\[
        \mathcal R_{(\eta)}
        \in\mathcal B(\mathcal D^p_\alpha)
        \Longleftrightarrow
        \mathcal R_{(\eta)}
        \in\mathcal K(\mathcal D^p_\alpha)
        \Longleftrightarrow
        F_\eta\in\mathcal D^p_\alpha.
\]
\end{enumerate}
\end{corollary}

\begin{proof}
Part~\textup{(i)} follows from
Theorem~\ref{thm:intro-above-critical}\textup{(ii)} by taking
\(q=p\) and \(\beta=\alpha\), so that \(s=1/p\), and then applying
Proposition~\ref{prop:dyadic-mean-lipschitz}.
Part~\textup{(ii)} follows from
Theorem~\ref{thm:intro-critical}\textup{(ii)} with
\(q=p\) and \(\beta=p-2\).
Part~\textup{(iii)} follows from
Theorem~\ref{thm:intro-below-critical} with
\(q=p\) and \(\beta=\alpha\).
\end{proof}

\section{Weighted Bergman and Hardy space results}
\label{sec:endpoint-regime}

Recall that \(A_a^p=\mathcal D^p_{a+p}\), with equivalent norms,
for \(a>-1\). Hence Theorems~\ref{thm:above-critical-Rhaly} and
\ref{thm:above-critical-compact} give the following characterization
for weighted Bergman spaces.

\begin{corollary}
\label{cor:mixed-Bergman}
Let \(1<p,q<\infty\) and \(a,b>-1\). Put
\(s=(a+2)/p-(b+1)/q\) and
\(A_j=2^{js}\|\Delta_jF_\eta\|_{H^q}\).

\begin{enumerate}[label=\textup{(\roman*)},leftmargin=2.2em]

\item If \(q<p\) and \(1/r=1/q-1/p\), then
\[
        \mathcal R_{(\eta)}\in\mathcal B(A_a^p,A_b^q)
        \Longleftrightarrow
        (A_j)\in\ell^r.
\]
Every bounded such operator is compact, and
\[
        \|\mathcal R_{(\eta)}\|
        \asymp_{p,q,a,b}
        \|(A_j)\|_{\ell^r}.
\]

\item If \(p\le q\), then
\[
        \mathcal R_{(\eta)}\in\mathcal B(A_a^p,A_b^q)
        \Longleftrightarrow
        \sup_{j\ge0}A_j<\infty,
\]
and
\[
        \mathcal R_{(\eta)}\in\mathcal K(A_a^p,A_b^q)
        \Longleftrightarrow
        \lim_{j\to\infty}A_j=0.
\]
Whenever \(\mathcal R_{(\eta)}\) is bounded,
\[
        \|\mathcal R_{(\eta)}\|
        \asymp_{p,q,a,b}
        \sup_{j\ge0}A_j,
        \qquad
        \|\mathcal R_{(\eta)}\|_e
        \asymp_{p,q,a,b}
        \limsup_{j\to\infty}A_j.
\]

\end{enumerate}
\end{corollary}

\begin{proof}
Apply Theorems~\ref{thm:above-critical-Rhaly} and
\ref{thm:above-critical-compact} with
\(\alpha=a+p\) and \(\beta=b+q\).
\end{proof}

\begin{remark}
For \(q=p\), Theorems~\ref{thm:above-critical-Rhaly} and
\ref{thm:above-critical-compact}, together with the standard dyadic
characterization of mean-Lipschitz spaces, recover
\cite[Theorem~4]{GalanopoulosGirela2026b} in the range
\[
        p-2<\alpha<\min\{p-1+\beta,2p-2\},
        \qquad
        \alpha-\beta+1>0.
\]
In the unweighted case \(a=b=0\),
Corollary~\ref{cor:mixed-Bergman} gives \(s=2/p-1/q\). Thus, for
\(1<p<q<\infty\) and \(2/p-1/q<1\), this agrees with the exponent
used in the proof of
\cite[Theorem~3]{GalanopoulosGirela2026b}.
\end{remark}

For Hardy spaces, the conditions
\(F_\eta\in\Lambda^p_{1/p}\) and
\(F_\eta\in\lambda^p_{1/p}\) are necessary for boundedness and
compactness of \(\mathcal R_{(\eta)}\) on \(H^p\), respectively, for
all \(1<p<\infty\)
\cite[Theorem~1(ii) and (iv)]{GalanopoulosGirela2026b}.
When \(1<p\le2\), they are also sufficient
\cite[Theorems~1--2]{GalanopoulosGirela2026}.
For \(p>2\), Galanopoulos and Girela proved sufficiency under
\(F_\eta\in\Lambda^r_{1/r}\) and
\(F_\eta\in\lambda^r_{1/r}\), respectively, for some \(2<r<p\)
\cite[Theorems~1--2]{GalanopoulosGirela2026}.
The next theorem shows that the endpoint choice \(r=p\) is not
sufficient in general.

\begin{theorem}
\label{thm:endpoint-failure}
Let \(2<p<\infty\). Then there exists a complex sequence
\(\eta=(\eta_n)_{n\ge0}\) such that
\[
        F_\eta\in H^\infty\cap\lambda^p_{1/p},
        \qquad
        \mathcal R_{(\eta)}\notin\mathcal B(H^p).
\]
\end{theorem}

\begin{proof}
Choose \(1/p<\gamma<1/2\). By
\cite[Theorem~I]{Rudin1959FourierCoefficients}, there exists a
sequence \((\varepsilon_k)_{k\ge1}\subset\{-1,1\}\) such that
\[
        \left\|\sum_{k=1}^{N}\varepsilon_kz^k\right\|_{H^\infty}
        \le 5N^{1/2},
        \qquad N\ge1.
\]
For \(j\ge1\) and \(n\in I_j\), set
\(\varepsilon_{j,n}=\varepsilon_{n-2^j+1}\), and put
\[
        P_j(z)
        =
        \sum_{n\in I_j}\varepsilon_{j,n}z^n
        =
        z^{2^j-1}\sum_{k=1}^{2^j}\varepsilon_kz^k.
\]
Then
\(\|P_j\|_{H^\infty}\le5\cdot2^{j/2}\).
Set \(\eta_0=\eta_1=0\) and, for \(n\in I_j\), \(j\ge1\), define
\[
        \eta_n
        =
        2^{-j(1/2+1/p)}
        (j+1)^{-(1/2-\gamma)}
        \varepsilon_{j,n}.
\]
Then
\[
        \|\Delta_jF_\eta\|_{H^\infty}
        \lesssim
        2^{-j/p}(j+1)^{-(1/2-\gamma)}.
\]
Since the right hand side is summable in \(j\),
\(F_\eta\in H^\infty\). Moreover,
\[
        2^{j/p}\|\Delta_jF_\eta\|_{H^p}
        \lesssim
        (j+1)^{-(1/2-\gamma)}
        \longrightarrow0, \qquad (j\to\infty).
\]
and hence
\(F_\eta\in\lambda^p_{1/p}\) by
Proposition~\ref{prop:dyadic-mean-lipschitz}.

We show that \(\mathcal R_{(\eta)}\) is not bounded on \(H^p\).
Let
\[
        a_n=n^{1/p-1}(\log n)^{-\gamma},
        \qquad n\ge2,
\]
and set
\(f(z)=\sum_{n=2}^\infty a_nz^n\).
Since \(p>2\),
\cite[Theorem~6.2.13]{JevticVukoticArsenovic2016} gives
\(f\in H^p\) whenever
\[
        \sum_{n=2}^{\infty}(n+1)^{p-2}a_n^p<\infty.
\]
Here
\[
        \sum_{n=2}^{\infty}(n+1)^{p-2}a_n^p
        \asymp_p
        \sum_{n=2}^{\infty}
        \frac{1}{n(\log n)^{\gamma p}}
        <\infty,
\]
because \(\gamma p>1\).

Put \(S_n=\sum_{k=2}^n a_k\). 
Comparing the sum with the corresponding integral gives
\[
        S_n
        \asymp
        \int_2^n x^{1/p-1}(\log x)^{-\gamma}\,dx
        \asymp
        n^{1/p}(\log n)^{-\gamma}.
\]
Hence, for \(n\in I_j\),
\[
        |\eta_nS_n|
        \asymp
        2^{-j/2}(j+1)^{-1/2}.
\]
Since \(|I_j|=2^j\),
\[
        \sum_{n\in I_j}|\eta_nS_n|^2
        \asymp
        \frac1{j+1},
\]
and therefore
\[
        \sum_{n=2}^{\infty}|\eta_nS_n|^2=\infty.
\]

Suppose that
\(\mathcal R_{(\eta)}\in\mathcal B(H^p)\), and let
\(f_M(z)=\sum_{k=2}^M a_kz^k\).
Since \(f_M\to f\) in \(H^p\),
\(\mathcal R_{(\eta)}f_M\to\mathcal R_{(\eta)}f\) in \(H^p\).
For each fixed \(n\ge2\), once \(M\ge n\),
\[
        \widehat{\mathcal R_{(\eta)}f_M}(n)
        =
        \eta_n\sum_{k=2}^n a_k
        =
        \eta_nS_n.
\]
Since the \(n\)-th Taylor coefficient functional is continuous on
\(H^p\), passing to the limit gives
\[
        \widehat{\mathcal R_{(\eta)}f}(n)
        =
        \eta_nS_n,
        \qquad n\ge2.
\]
Since \((\eta_nS_n)_{n\ge2}\notin\ell^2\),
\(\mathcal R_{(\eta)}f\notin H^2\), contradicting
\(\mathcal R_{(\eta)}f\in H^p\subset H^2\).
\end{proof}

A complete characterization of boundedness and compactness on
\(H^p\), \(p>2\), for arbitrary complex symbols remains open.

\section{Ces\`aro-type operators induced by positive measures}
\label{sec:cesaro-type}

Recall that, for a finite Borel measure \(\mu\) on \([0,1)\),
\(C_\mu=\mathcal R_{(\mu_n)}\), where
\(
        \mu_n=\int_{[0,1)}t^n\,d\mu(t).
\)
If \(\mu\) is positive, then \((\mu_n)\) is nonnegative and
nonincreasing.
Consequently, the boundedness and compactness characterizations obtained
above for \(\mathcal R_{(\eta)}\), applied to
\(C_\mu=\mathcal R_{(\mu_n)}\), can be expressed directly in terms of
the moment sequence \((\mu_n)\).
The coefficient estimates used below are collected in
Appendix~\ref{app:monotone-reductions}.

\begin{corollary}
\label{cor:cesaro-positive-moment}
Let \(\mu\) be a finite positive Borel measure on \([0,1)\), put
\(\mu_n=\int_{[0,1)}t^n\,d\mu(t)\), and let \(1<p<\infty\). Then the
following hold:

\begin{enumerate}[label=\textup{(\roman*)},leftmargin=2.2em]

\item On each of \(H^p\), \(A_a^p\) with \(a>-1\), and
\(\mathcal D^p_\alpha\) with \(\alpha>p-2\), the operator \(C_\mu\) is
bounded if and only if \(\mu_n=O((n+1)^{-1})\), and compact if and only
if \(\mu_n=o((n+1)^{-1})\).

\item Let \(1<q<\infty\) and \(\beta>-1\). For
\(\mathcal D^p_{p-2}\), we have:
\begin{enumerate}[label=\textup{(\alph*)},leftmargin=2.0em]

\item If \(q<p\) and \(1/r=1/q-1/p\), then
\[
\begin{aligned}
C_\mu\in
\mathcal B(\mathcal D^p_{p-2},\mathcal D^q_\beta)
&\Longleftrightarrow
C_\mu\in
\mathcal K(\mathcal D^p_{p-2},\mathcal D^q_\beta) \\
&\Longleftrightarrow
\sum_{J\ge1}(J+1)^{r/q'}
\left(
\sum_{n\ge2^J}
n^{2q-\beta-3}\mu_n^q
\right)^{r/q}<\infty \\
&\Longleftrightarrow
\sum_{N=2}^{\infty}
\frac{(\log N)^{r/q'}}{N}
\left(
\sum_{n\ge N}
n^{2q-\beta-3}\mu_n^q
\right)^{r/q}<\infty.
\end{aligned}
\]

\item If \(p\le q\), then
\[
\begin{aligned}
C_\mu\in
\mathcal B(\mathcal D^p_{p-2},\mathcal D^q_\beta)
&\Longleftrightarrow
\sup_{J\ge1}(J+1)^{q/p'}
\sum_{n\ge2^J}
n^{2q-\beta-3}\mu_n^q<\infty \\
&\Longleftrightarrow
\sup_{N\ge2}
(\log N)^{q/p'}
\sum_{n\ge N}
n^{2q-\beta-3}\mu_n^q<\infty,
\end{aligned}
\]
and
\[
\begin{aligned}
C_\mu\in
\mathcal K(\mathcal D^p_{p-2},\mathcal D^q_\beta)
&\Longleftrightarrow
\lim_{J\to\infty}(J+1)^{q/p'}
\sum_{n\ge2^J}
n^{2q-\beta-3}\mu_n^q=0 \\
&\Longleftrightarrow
\lim_{N\to\infty}
(\log N)^{q/p'}
\sum_{n\ge N}
n^{2q-\beta-3}\mu_n^q=0.
\end{aligned}
\]
Whenever \(C_\mu\) is bounded,
\[
\begin{aligned}
\|C_\mu\|_e^q
&\asymp_{p,q,\beta}
\limsup_{J\to\infty}(J+1)^{q/p'}
\sum_{n\ge2^J}
n^{2q-\beta-3}\mu_n^q \\
&\asymp_{p,q,\beta}
\limsup_{N\to\infty}
(\log N)^{q/p'}
\sum_{n\ge N}
n^{2q-\beta-3}\mu_n^q.
\end{aligned}
\]

\end{enumerate}

\item Let \(1<q<\infty\), \(\beta>-1\), and
\(-1<\alpha<p-2\). Then
\[
\begin{aligned}
        C_\mu\in
        \mathcal B(\mathcal D^p_\alpha,\mathcal D^q_\beta)
        &\Longleftrightarrow
        C_\mu\in
        \mathcal K(\mathcal D^p_\alpha,\mathcal D^q_\beta) \\
        &\Longleftrightarrow
        F_\mu\in\mathcal D^q_\beta \\
        &\Longleftrightarrow
        \sum_{n=1}^{\infty}
        n^{2q-\beta-3}\mu_n^q<\infty.
\end{aligned}
\]

\item Let \(1<q<\infty\), \(\beta>-1\), \(\alpha>p-2\), and put
\(s=(\alpha+2)/p-(\beta+1)/q\).
\begin{enumerate}[label=\textup{(\alph*)},leftmargin=2.0em]

\item If \(q<p\) and \(1/r=1/q-1/p\), then
\[
\begin{aligned}
        C_\mu\in
        \mathcal B(\mathcal D^p_\alpha,\mathcal D^q_\beta)
        &\Longleftrightarrow
        C_\mu\in
        \mathcal K(\mathcal D^p_\alpha,\mathcal D^q_\beta) \\
        &\Longleftrightarrow
        \sum_{j\ge0}
        \left(
        \sum_{n=2^j}^{2^{j+1}-1}
        n^{q+sq-2}\mu_n^q
        \right)^{r/q}<\infty.
\end{aligned}
\]

\item If \(p\le q\), then
\[
        C_\mu\in
        \mathcal B(\mathcal D^p_\alpha,\mathcal D^q_\beta)
        \Longleftrightarrow
        \sup_{j\ge0}
        \sum_{n=2^j}^{2^{j+1}-1}
        n^{q+sq-2}\mu_n^q<\infty,
\]
and
\[
        C_\mu\in
        \mathcal K(\mathcal D^p_\alpha,\mathcal D^q_\beta)
        \Longleftrightarrow
        \lim_{j\to\infty}
        \sum_{n=2^j}^{2^{j+1}-1}
        n^{q+sq-2}\mu_n^q
        =0.
\]
Whenever \(C_\mu\) is bounded,
\[
        \|C_\mu\|_e^q
        \asymp_{p,q,\alpha,\beta}
        \limsup_{j\to\infty}
        \sum_{n=2^j}^{2^{j+1}-1}
        n^{q+sq-2}\mu_n^q.
\]

\end{enumerate}

\end{enumerate}
\end{corollary}

\begin{proof}
For positive \(\mu\), the sequence \((\mu_n)\) is nonnegative and
nonincreasing, and \(\mu_0=\mu([0,1))<\infty\).

For \(H^p\), part~\textup{(i)} follows from
\cite[Theorems~1--2 and Lemma~2]{GalanopoulosGirelaMerchan2022}.
For \(A_a^p\), it follows from
Corollary~\ref{cor:mixed-Bergman} together with
Lemma~\ref{lem:app-monotone-lambda}. For
\(\mathcal D^p_\alpha\), \(\alpha>p-2\), it follows from
Corollary~\ref{cor:same-space-classification} and
Lemma~\ref{lem:app-monotone-lambda}.

Part~\textup{(ii)} follows from
Theorem~\ref{thm:intro-critical} and
Lemma~\ref{lem:app-positive-measure-reductions}\textup{(i)}.
Since \(\mu_0<\infty\), the terms \(\mu_0^r\) and \(\mu_0^q\) may be
omitted from the finiteness conditions. The limit superior equivalences in
Lemma~\ref{lem:app-positive-measure-reductions}\textup{(i)(b)}
do not involve \(\mu_0\).

Part~\textup{(iii)} follows from
Theorem~\ref{thm:intro-below-critical} and
Lemma~\ref{lem:app-positive-monotone-coefficients}, applied with
\((q,\beta)\).

Part~\textup{(iv)} follows from
Theorem~\ref{thm:intro-above-critical} and
Lemma~\ref{lem:app-positive-measure-reductions}\textup{(ii)}.
Since \(\mu_0<\infty\), the term \(\mu_0^r\) in
Lemma~\ref{lem:app-positive-measure-reductions}\textup{(ii)(a)}
and the term \(\mu_0\) in
Lemma~\ref{lem:app-positive-measure-reductions}\textup{(ii)(b)}
do not affect the boundedness conditions, while the limit superior in
Lemma~\ref{lem:app-positive-measure-reductions}\textup{(ii)(b)}
contains no term involving \(\mu_0\).
\end{proof}

\begin{remark}
In the special case of positive measures, equal Bergman weights,
and \(q<p\), the boundedness assertion in
Corollary~\ref{cor:mixed-Bergman}\textup{(i)} is also contained in
\cite[Theorem~1.2]{PanTongYang2026}, with kernel order \(1\), where
an equivalent condition is formulated in terms of the tail function
of the inducing measure.
\end{remark}

\begin{remark}
\label{rem:dirichlet-cesaro-comparison}
For the analytic Besov space \(B^p=\mathcal D^p_{p-2}\),
Sun, Ye and Zhou \cite{SunYeZhou2025} characterized boundedness and
compactness of \(C_\mu:B^p\to X\) for Banach spaces \(X\) satisfying
\(\Lambda_{1/s}^s\subset X\subset\mathcal B\), \(s>1\). They left open
the characterization of boundedness and compactness of
\(C_\mu:B^p\to B^p\), \(p>1\), a problem later noted again by Tang
\cite{Tang2025Besov}. 
Taking \(q=p\) and \(\beta=p-2\) in
Corollary~\ref{cor:cesaro-positive-moment}\textup{(ii)} answers this
question explicitly. Set
\[
        L_N=(\log N)^{p-1}\sum_{n\ge N}n^{p-1}\mu_n^p,
        \qquad N\ge2.
\]
Then
\[
\begin{aligned}
        C_\mu:B^p\to B^p \text{ is bounded}
        &\iff \sup_{N\ge2}L_N<\infty,\\
        C_\mu:B^p\to B^p \text{ is compact}
        &\iff L_N\to0,
\end{aligned}
\qquad
        \|C_\mu\|_e^p
        \asymp
        \limsup_{N\to\infty}L_N.
\]
For \(p=2\), these conditions reduce to the known tail conditions for
the classical Dirichlet space; see
\cite{BaoGuoSunWang2024,BlascoGalanopoulosGirela2026}. In the range
\(p>\alpha+2\), Xie, Liu and Lin \cite{XieLiuLin2026} studied
Ces\`aro-type operators induced by positive measures from
\(\mathcal D^p_\alpha\) into \(\mathcal D^q_\beta\); see
\cite[Theorem~1.2]{LinLiuTangXie2026} for the corrected
characterization. Corollary~\ref{cor:cesaro-positive-moment}
\textup{(iii)} recovers this corrected characterization in the
equivalent form \(F_\mu\in\mathcal D^q_\beta\), while
part~\textup{(iv)} gives the characterization for \(\alpha>p-2\) in
terms of dyadic coefficient blocks.
\end{remark}

\appendix
\section{Monotone coefficient estimates}
\label{app:monotone-reductions}

We collect the coefficient estimates for nonnegative nonincreasing
sequences used in Section~\ref{sec:cesaro-type}.

The first equivalence in the following lemma is
\cite[Proposition~A]{GalanopoulosGirela2026}. We include a short
proof which also gives the characterization of \(\lambda^p_{1/p}\).

\begin{lemma}
\label{lem:app-monotone-lambda}
Let \(1<p<\infty\), and let
\(F(z)=\sum_{n\ge0}a_nz^n\), where \((a_n)\) is nonnegative and
nonincreasing. Then
\[
\begin{aligned}
F\in\Lambda^p_{1/p}
&\Longleftrightarrow
a_n=O((n+1)^{-1}), \\
F\in\lambda^p_{1/p}
&\Longleftrightarrow
a_n=o((n+1)^{-1}).
\end{aligned}
\]
\end{lemma}

\begin{proof}
Let \(N=2^j\), \(j\ge1\). Writing
\(
        \Delta_jF(z)
        =
        z^N\sum_{m=0}^{N-1}a_{N+m}z^m,
\) and applying
\cite[Theorem~6.2.14]{JevticVukoticArsenovic2016} in its finite
polynomial form, we obtain
\[
        \|\Delta_jF\|_{H^p}^p
        \asymp_p
        \sum_{m=0}^{N-1}(m+1)^{p-2}a_{N+m}^p.
\]
Since \((a_n)\) is nonincreasing and
\(\sum_{m=0}^{N-1}(m+1)^{p-2}\asymp_pN^{p-1}\),
\[
        Na_{2N}
        \lesssim_p
        N^{1/p}\|\Delta_jF\|_{H^p}
        \lesssim_p
        Na_N.
\]
Hence Proposition~\ref{prop:dyadic-mean-lipschitz} gives the
sufficiency of both coefficient conditions.

Conversely, the same proposition and the lower estimate above give
\[
        Na_{2N}=O(1)
        \quad\text{if }F\in\Lambda^p_{1/p},
        \qquad
        Na_{2N}=o(1)
        \quad\text{if }F\in\lambda^p_{1/p},
\]
for \(N=2^j\). Given sufficiently large \(n\), choose \(N=2^j\)
such that \(2N\le n<4N\). By monotonicity,
\[
        na_n\le4Na_{2N}.
\]
Thus \(na_n=O(1)\), respectively \(na_n=o(1)\), which proves the
result.
\end{proof}

\begin{lemma}
\label{lem:app-dyadic-weighted-condensation}
Let \(1<p<\infty\), and let \((b_n)_{n\ge0}\) be nonnegative and
nonincreasing. Put \(g(z)=\sum_{n\ge0}b_nz^n\) and, for \(j\ge0\),
\[
        A_j=2^j\|\Delta_jg\|_{H^p}^p,
        \qquad
        R_j=\sum_{n=2^j}^{2^{j+1}-1}n^{p-1}b_n^p.
\]
Then, for every \(j\ge1\),
\[
        A_j\lesssim_pR_{j-1}+R_j,
        \qquad
        R_j\lesssim_pA_{j-1}+A_j.
\]
\end{lemma}

\begin{proof}
Fix \(j\ge1\) and put \(N=2^j\). By
\cite[Theorem~6.2.14]{JevticVukoticArsenovic2016},
\[
        A_j
        \asymp_p
        N\sum_{k=0}^{N-1}(k+1)^{p-2}b_{N+k}^p,
        \qquad
        R_j
        \asymp_p
        N^{p-1}\sum_{k=0}^{N-1}b_{N+k}^p.
\]

If \(p\ge2\), then
\((k+1)^{p-2}\le N^{p-2}\), so \(A_j\lesssim_pR_j\).
If \(1<p<2\), then by monotonicity,
\[
        A_j
        \lesssim_p
        Nb_N^p\sum_{k=0}^{N-1}(k+1)^{p-2}
        \lesssim_p
        N^pb_N^p.
\]
Moreover, since \(b_n\ge b_N\) and \(n\asymp N\) for
\(N/2\le n<N\),
\[
        R_{j-1}
        \gtrsim_p
        N^{p-1}\sum_{n=N/2}^{N-1}b_N^p
        \asymp_p N^pb_N^p.
\]
Hence \(A_j\lesssim_pR_{j-1}\).
This proves \(A_j\lesssim_pR_{j-1}+R_j\).

For the reverse estimate, suppose first that \(1<p\le2\). Since
\((k+1)^{p-2}\ge N^{p-2}\) for \(0\le k<N\), and
\(N+k\asymp N\), we obtain
\[
        R_j\lesssim_p A_j.
\]
Suppose now that \(p>2\). Splitting the sum at \(N/2\), we have
\[
\begin{aligned}
        R_j
        &\lesssim_p
        N^{p-1}\sum_{k<N/2}b_{N+k}^p
        +
        N^{p-1}\sum_{N/2\le k<N}b_{N+k}^p \\
        &\lesssim_p
        N^pb_N^p+A_j.
\end{aligned}
\]
For \(j\ge2\), we have \(N=2^j\ge4\). For
\(N/4\le k<N/2\), \(k+1\asymp N\) and
\(b_{N/2+k}\ge b_N\). Hence
\[
        A_{j-1}
        \gtrsim_p
        N\sum_{N/4\le k<N/2}
        (k+1)^{p-2}b_{N/2+k}^p
        \gtrsim_p
        N^pb_N^p.
\]
Thus
\[
        R_j\lesssim_p A_{j-1}+A_j,
        \qquad j\ge2.
\]
For \(j=1\), monotonicity and the definition of \(A_0\) give
\[
        R_1\lesssim_p b_1^p\lesssim_p A_0.
\]
The proof is complete.
\end{proof}

\begin{lemma}
\label{lem:app-positive-monotone-coefficients}
Let \(1<p<\infty\), \(\beta>-1\), and let
\(g(z)=\sum_{n\ge0}a_nz^n\), where \((a_n)\) is nonnegative and
nonincreasing. Then
\[
        \|g'\|_{A^p_\beta}^p
        \asymp_{p,\beta}
        a_1^p+
        \sum_{n=2}^{\infty}n^{2p-\beta-3}a_n^p.
\]
\end{lemma}

\begin{proof}
Applying Proposition~\ref{prop:dirichlet-dyadic-model} with
\(\alpha=\beta\) to \(g-g(0)\), we obtain
\[
        \|g'\|_{A^p_\beta}^p
        \asymp_{p,\beta}
        |a_1|^p+
        \sum_{j\ge1}
        2^{j(p-\beta-1)}
        \|\Delta_jg\|_{H^p}^p.
\]
For \(j\ge1\), put
\[
        A_j=2^j\|\Delta_jg\|_{H^p}^p,
        \qquad
        R_j=\sum_{n=2^j}^{2^{j+1}-1}n^{p-1}a_n^p.
\]
Since \(I_1=\{2,3\}\), we have
\(A_1=2\|\Delta_1g\|_{H^p}^p\asymp_p a_2^p+a_3^p\asymp_p R_1\).
For \(j\ge2\), since
\(
        2^{(j-1)(p-\beta-2)}
        \asymp_{p,\beta}
        2^{j(p-\beta-2)},
\)
Lemma~\ref{lem:app-dyadic-weighted-condensation} gives
\[
\begin{aligned}
        \sum_{j\ge2}2^{j(p-\beta-2)}A_j
        &\lesssim_p
        \sum_{j\ge2}2^{j(p-\beta-2)}
        (R_{j-1}+R_j) \\
        &\lesssim_{p,\beta}
        \sum_{j\ge1}2^{j(p-\beta-2)}R_j.
\end{aligned}
\]
The reverse estimate follows in the same way from
\(R_j\lesssim_p A_{j-1}+A_j\). Hence
\[
        \sum_{j\ge1}2^{j(p-\beta-2)}A_j
        \asymp_{p,\beta}
        \sum_{j\ge1}2^{j(p-\beta-2)}R_j.
\]
Since \(n\asymp2^j\) for \(2^j\le n<2^{j+1}\),
\[
\begin{aligned}
        \sum_{j\ge1}2^{j(p-\beta-2)}R_j
        &=
        \sum_{j\ge1}
        \sum_{n=2^j}^{2^{j+1}-1}
        2^{j(p-\beta-2)}n^{p-1}a_n^p \\
        &\asymp_{p,\beta}
        \sum_{n=2}^{\infty}n^{2p-\beta-3}a_n^p.
\end{aligned}
\]
Combining the preceding estimates yields
\[
        \|g'\|_{A^p_\beta}^p
        \asymp_{p,\beta}
        |a_1|^p+
        \sum_{n=2}^{\infty}n^{2p-\beta-3}a_n^p,
\]
which proves the lemma.\end{proof}

\begin{lemma}
\label{lem:app-positive-measure-reductions}
Let \(1<p<\infty\). Let \(\mu\) be a finite positive Borel measure on
\([0,1)\), and put
\[
        \mu_n=\int_{[0,1)}t^n\,d\mu(t),
        \qquad
        F_\mu(z)=\sum_{n=0}^{\infty}\mu_nz^n.
\]
Then the following hold.

\begin{enumerate}[label=\textup{(\roman*)},leftmargin=2.2em]

\item Let \(1<q<\infty\) and \(\beta>-1\), and put
\[
        U_J
        =
        \sum_{j\ge J}
        2^{j(q-\beta-1)}
        \|\Delta_jF_\mu\|_{H^q}^q,
        \qquad
        V_J
        =
        \sum_{n\ge2^J}
        n^{2q-\beta-3}\mu_n^q,
\]
and
\[
        T_N
        =
        \sum_{n\ge N}
        n^{2q-\beta-3}\mu_n^q.
\]

\begin{enumerate}[label=\textup{(\alph*)},leftmargin=2.0em]

\item If \(q<p\) and \(1/r=1/q-1/p\), then
\[
\begin{aligned}
        \sum_{J\ge0}(J+1)^{r/q'}U_J^{r/q}
        &\asymp_{p,q,\beta}
        \mu_0^r+
        \sum_{J\ge1}(J+1)^{r/q'}V_J^{r/q} \\
        &\asymp_{p,q,\beta}
        \mu_0^r+
        \sum_{N=2}^{\infty}
        \frac{(\log N)^{r/q'}}{N}T_N^{r/q}.
\end{aligned}
\]

\item If \(p\le q\), then
\[
\begin{aligned}
        \sup_{J\ge0}(J+1)^{q/p'}U_J
        &\asymp_{p,q,\beta}
        \mu_0^q+
        \sup_{J\ge1}(J+1)^{q/p'}V_J \\
        &\asymp_{p,q,\beta}
        \mu_0^q+
        \sup_{N\ge2}(\log N)^{q/p'}T_N.
\end{aligned}
\]
Moreover,
\[
\begin{aligned}
        \limsup_{J\to\infty}(J+1)^{q/p'}U_J
        &\asymp_{p,q,\beta}
        \limsup_{J\to\infty}(J+1)^{q/p'}V_J \\
        &\asymp_{p,q,\beta}
        \limsup_{N\to\infty}(\log N)^{q/p'}T_N.
\end{aligned}
\]

\end{enumerate}

\item Let \(1<q<\infty\), \(\beta>-1\), and \(\alpha>p-2\). Put
\(s=(\alpha+2)/p-(\beta+1)/q\), and define
\[
        X_j=2^{js}\|\Delta_jF_\mu\|_{H^q},
        \qquad
        Y_j=
        \left(
        \sum_{n=2^j}^{2^{j+1}-1}
        n^{q+sq-2}\mu_n^q
        \right)^{1/q}.
\]

\begin{enumerate}[label=\textup{(\alph*)},leftmargin=2.0em]

\item If \(q<p\) and \(1/r=1/q-1/p\), then
\[
        \sum_{j\ge0}X_j^r
        \asymp_{p,q,\alpha,\beta}
        \mu_0^r+\sum_{j\ge0}Y_j^r.
\]

\item If \(p\le q\), then
\[
        \sup_{j\ge0}X_j
        \asymp_{p,q,\alpha,\beta}
        \mu_0+\sup_{j\ge0}Y_j,
\]
and
\[
        \limsup_{j\to\infty}X_j
        \asymp_{p,q,\alpha,\beta}
        \limsup_{j\to\infty}Y_j.
\]

\end{enumerate}
\end{enumerate}
\end{lemma}

\begin{proof}
Since \(\mu\) is positive, \((\mu_n)\) is nonnegative and
nonincreasing.

For \textup{(i)}, put
\[
        A_j
        =
        2^{j(q-\beta-1)}
        \|\Delta_jF_\mu\|_{H^q}^q,
        \qquad
        R_j
        =
        \sum_{n=2^j}^{2^{j+1}-1}
        n^{2q-\beta-3}\mu_n^q.
\]
Lemma~\ref{lem:app-dyadic-weighted-condensation}, applied with
exponent \(q\), gives
\begin{equation}
\label{eq:app-positive-AR}
        A_j\lesssim_{q,\beta}R_{j-1}+R_j,
        \qquad
        R_j\lesssim_{q,\beta}A_{j-1}+A_j,
        \qquad j\ge1.
\end{equation}
Moreover,
\[
        A_0=\|\mu_0+\mu_1z\|_{H^q}^q\asymp_q\mu_0^q,
        \qquad
        R_0=\mu_1^q\le\mu_0^q.
\]
Since
\[
        U_J=\sum_{j\ge J}A_j,
        \qquad
        V_J=\sum_{j\ge J}R_j=T_{2^J},
\]
summing \eqref{eq:app-positive-AR} gives
\begin{equation}
\label{eq:app-positive-UV}
        U_J\lesssim_{q,\beta}V_{J-1},
        \qquad
        V_J\lesssim_{q,\beta}U_{J-1},
        \qquad J\ge1.
\end{equation}

If \(2^J\le N<2^{J+1}\), \(J\ge1\), then
\begin{equation}
\label{eq:app-positive-TN}
        V_{J+1}\le T_N\le V_J,
        \qquad
        \log N\asymp J+1,
        \qquad
        \sum_{m=2^J}^{2^{J+1}-1}\frac1m\asymp1.
\end{equation}

Suppose first that \(q<p\) and \(1/r=1/q-1/p\). Since
\((J+1)^{r/q'}\asymp_{p,q}J^{r/q'}\) for \(J\ge1\),
\eqref{eq:app-positive-UV} gives
\[
        \sum_{J\ge1}(J+1)^{r/q'}U_J^{r/q}
        \lesssim_{p,q,\beta}
        V_0^{r/q}+
        \sum_{J\ge1}(J+1)^{r/q'}V_J^{r/q},\]
        and 
       \[
        \sum_{J\ge1}(J+1)^{r/q'}V_J^{r/q}
        \lesssim_{p,q,\beta}
        \sum_{J\ge0}(J+1)^{r/q'}U_J^{r/q}.
\]
Since
\[
        U_0=A_0+U_1
        \lesssim_{q,\beta}
        \mu_0^q+V_0
        \lesssim_{q,\beta}
        \mu_0^q+V_1,
\]
and \(\mu_0^q\lesssim_q U_0\), it follows that
\begin{equation}
\label{eq:app-positive-UV-sum}
        \sum_{J\ge0}(J+1)^{r/q'}U_J^{r/q}
        \asymp_{p,q,\beta}
        \mu_0^r+
        \sum_{J\ge1}(J+1)^{r/q'}V_J^{r/q}.
\end{equation}

For \(2^J\le N<2^{J+1}\), \eqref{eq:app-positive-TN} gives
\[
\begin{aligned}
        (J+1)^{r/q'}V_{J+1}^{r/q}
        &\lesssim_{p,q}
        \sum_{N=2^J}^{2^{J+1}-1}
        \frac{(\log N)^{r/q'}}{N}T_N^{r/q}\\
        &\lesssim_{p,q}
        (J+1)^{r/q'}V_J^{r/q}.
\end{aligned}
\]
Summing over \(J\ge1\) and using \(T_2=V_1\), we obtain
\begin{equation}
\label{eq:app-positive-VT-sum}
        \sum_{J\ge1}(J+1)^{r/q'}V_J^{r/q}
        \asymp_{p,q}
        \sum_{N=2}^{\infty}
        \frac{(\log N)^{r/q'}}{N}T_N^{r/q}.
\end{equation}
Combining \eqref{eq:app-positive-UV-sum} and
\eqref{eq:app-positive-VT-sum} proves \textup{(i)(a)}.

Suppose now that \(p\le q\), and put \(b=q/p'\). Since
\((J+1)^b\asymp_{p,q}J^b\) for \(J\ge1\),
\eqref{eq:app-positive-UV} gives
\[
        \sup_{J\ge1}(J+1)^bU_J
        \lesssim_{p,q,\beta}
        V_0+
        \sup_{J\ge1}(J+1)^bV_J,
\]
and
\[
        \sup_{J\ge1}(J+1)^bV_J
        \lesssim_{p,q,\beta}
        \sup_{J\ge0}(J+1)^bU_J.
\]
Since \(V_0=R_0+V_1\), \(R_0\le\mu_0^q\), and
\(A_0\asymp_q\mu_0^q\), we obtain
\begin{equation}
\label{eq:app-positive-UV-sup}
        \sup_{J\ge0}(J+1)^bU_J
        \asymp_{p,q,\beta}
        \mu_0^q+
        \sup_{J\ge1}(J+1)^bV_J.
\end{equation}
For \(2^J\le N<2^{J+1}\), \eqref{eq:app-positive-TN} gives
\[
        (J+1)^bV_{J+1}
        \lesssim_{p,q}
        \sup_{2^J\le N<2^{J+1}}(\log N)^bT_N
        \lesssim_{p,q}
        (J+1)^bV_J.
\]
Taking the supremum over \(J\ge1\), we obtain
\begin{equation}
\label{eq:app-positive-VT-sup}
        \sup_{J\ge1}(J+1)^bV_J
        \asymp_{p,q}
        \sup_{N\ge2}(\log N)^bT_N.
\end{equation}
Combining \eqref{eq:app-positive-UV-sup} and
\eqref{eq:app-positive-VT-sup} proves the first assertion in
\textup{(i)(b)}.

For the limit superior, \eqref{eq:app-positive-UV} gives
\[
        \limsup_{J\to\infty}(J+1)^bU_J
        \asymp_{p,q,\beta}
        \limsup_{J\to\infty}(J+1)^bV_J,
\]
while \eqref{eq:app-positive-TN} gives
\[
        \limsup_{J\to\infty}(J+1)^bV_J
        \asymp_{p,q}
        \limsup_{N\to\infty}(\log N)^bT_N.
\]
Combining the last two estimates proves the limit superior statement
in \textup{(i)(b)}.

For \textup{(ii)}, apply
Lemma~\ref{lem:app-dyadic-weighted-condensation} with exponent \(q\)
and multiply its two estimates by \(2^{j(sq-1)}\). Since
\(2^{(j-1)(sq-1)}\asymp_{p,q,\alpha,\beta}2^{j(sq-1)}\) and
\(n\asymp2^j\) for \(2^j\le n<2^{j+1}\), taking \(q\)-th roots gives
\begin{equation}
\label{eq:app-positive-XY}
        X_j
        \lesssim_{p,q,\alpha,\beta}
        Y_{j-1}+Y_j,
        \qquad
        Y_j
        \lesssim_{p,q,\alpha,\beta}
        X_{j-1}+X_j,
        \qquad j\ge1.
\end{equation}
Also,
\[
        X_0=\|\mu_0+\mu_1z\|_{H^q}\asymp_q\mu_0,
        \qquad
        Y_0=\mu_1\le\mu_0.
\]

If \(q<p\), then \eqref{eq:app-positive-XY} gives
\[
        \sum_{j\ge1}X_j^r
        \lesssim_{p,q,\alpha,\beta}
        \sum_{j\ge0}Y_j^r,
        \qquad\text{and}\qquad
        \sum_{j\ge1}Y_j^r
        \lesssim_{p,q,\alpha,\beta}
        \sum_{j\ge0}X_j^r.
\]
Since \(X_0\asymp_q\mu_0\) and \(Y_0\le\mu_0\), we obtain
\[
        \sum_{j\ge0}X_j^r
        \asymp_{p,q,\alpha,\beta}
        \mu_0^r+\sum_{j\ge0}Y_j^r.
\]
This proves \textup{(ii)(a)}.

If \(p\le q\), \eqref{eq:app-positive-XY} gives
\[
        \sup_{j\ge1}X_j
        \lesssim_{p,q,\alpha,\beta}
        \sup_{j\ge0}Y_j,
        \qquad\text{and}\qquad
        \sup_{j\ge1}Y_j
        \lesssim_{p,q,\alpha,\beta}
        \sup_{j\ge0}X_j.
\]
Using \(X_0\asymp_q\mu_0\) and \(Y_0\le\mu_0\), we obtain
\[
        \sup_{j\ge0}X_j
        \asymp_{p,q,\alpha,\beta}
        \mu_0+\sup_{j\ge0}Y_j.
\]
Finally, \eqref{eq:app-positive-XY} gives
\[
        \limsup_{j\to\infty}X_j
        \lesssim_{p,q,\alpha,\beta}
        \limsup_{j\to\infty}Y_j,
        \qquad\text{and}\qquad
        \limsup_{j\to\infty}Y_j
        \lesssim_{p,q,\alpha,\beta}
        \limsup_{j\to\infty}X_j.
\]
Hence
\[
        \limsup_{j\to\infty}X_j
        \asymp_{p,q,\alpha,\beta}
        \limsup_{j\to\infty}Y_j.
\]
This proves \textup{(ii)(b)}.
\end{proof}

\medskip
\noindent\textbf{Funding.}\enspace
The author was supported by the Department of Education of Guangdong
Province (Grant No.~2023KTSCX072), the Guangdong Basic and Applied Basic
Research Foundation (Grant No.~2024A1515012551), and Lingnan Normal
University (Grant No.~LT2410).

\smallskip
\noindent\textbf{Declaration of competing interest.}\enspace
The author declares that there are no known competing financial
interests or personal relationships that could have appeared to
influence the work reported in this paper.

\smallskip
\noindent\textbf{Data availability.}\enspace
No data was used for the research described in this article.

\smallskip
\noindent\textbf{Declaration of generative AI and AI-assisted
technologies in the manuscript preparation process.}\enspace
During the preparation of this manuscript, the author used OpenAI's
ChatGPT for English-language editing, organization of the exposition,
preliminary consistency checks of the presentation and references, and
assistance with the construction of the counterexample in
Theorem~\ref{thm:endpoint-failure}. The author independently reviewed
and verified all AI-assisted content and takes full responsibility for
the content of the manuscript.

\end{document}